\documentclass[12pt]{article}
\usepackage{amssymb, amsmath, amsthm, amscd}
\usepackage[dvips]{graphics}
\usepackage[utf8]{inputenc}
\usepackage[all,cmtip]{xy}
\usepackage{bbm}
\usepackage{enumitem}
\usepackage{setspace}
\usepackage[colorlinks=true]{hyperref}
\hypersetup{colorlinks   = true}
\hypersetup{linkcolor=blue}
\newenvironment{Acknowledgements}{Acknowledgements.}

\begin{document}

\newtheorem{The}{Theorem}[section]
\newtheorem{Lem}[The]{Lemma}
\newtheorem{Prop}[The]{Proposition}
\newtheorem{Cor}[The]{Corollary}
\newtheorem{Rem}[The]{Remark}
\newtheorem{Obs}[The]{Observation}
\newtheorem{SConj}[The]{Standard Conjecture}
\newtheorem{Titre}[The]{\!\!\!\! }
\newtheorem{Conj}[The]{Conjecture}
\newtheorem{Question}[The]{Question}
\newtheorem{Prob}[The]{Problem}
\newtheorem{Def}[The]{Definition}
\newtheorem{Not}[The]{Notation}
\newtheorem{Claim}[The]{Claim}
\newtheorem{Conc}[The]{Conclusion}
\newtheorem{Ex}[The]{Example}
\newtheorem{Fact}[The]{Fact}
\newtheorem{Formula}[The]{Formula}
\newtheorem{Formulae}[The]{Formulae}
\newcommand{\C}{\mathbb{C}}
\newcommand{\R}{\mathbb{R}}
\newcommand{\N}{\mathbb{N}}
\newcommand{\Z}{\mathbb{Z}}
\newcommand{\Q}{\mathbb{Q}}
\newcommand{\Proj}{\mathbb{P}}
\newcommand{\Rc}{\mathcal{R}}
\newcommand{\Oc}{\mathcal{O}}

\begin{center}

{\Large\bf  Co-polarised Deformations of Gauduchon Calabi-Yau $\partial\bar{\partial}$-Manifolds and Deformation of $p$-SKT $h$-$\partial\bar{\partial}$-Manifolds}

\end{center}

\begin{center}

{\large Houda Bellitir}

\end{center}

\vspace{1ex}

\begin{abstract}
The main result of this paper is to study the local deformations of Calabi-Yau $\partial\bar{\partial}$-manifold that are co-polarised by the Gauduchon metric by considering the subfamily of co-polarised fibres by the class of Aeppli/De Rham-Gauduchon cohomology of Gauduchon metric given at the beginning on the central fibre. In the latter part, we prove that the $p$-SKT $h$-$\partial\bar{\partial}$-property is deformation open by constructing and studying a new notion called $hp$-Hermitian symplectic ($hp$-HS) form.
\end{abstract}

\noindent {\bf keywords. } $\partial\bar{\partial}$-manifold, deformations of complex structures, co-polarisation by Gauduchon class, primitive class, Weil-Petersson metric, $h$-$\partial\bar{\partial}$-manifold, $p$-SKT manifold. 
 
\section{Introduction}
Suppose that $X$ is a compact complex $n$-dimensional manifold. A Hermitian metric (i.e. a $C^\infty$ positive definite $(1,1)$-form) $\omega>0$ is said to be Gauduchon metric \cite{Gau77} if: 
$$\partial\bar{\partial} \omega^{n-1}=0. $$
It is well known that the Gauduchon metrics always exist \cite{Gau77} on a compact complex manifold.\\

Recall that a compact complex manifold $X$ of dimension $n$ is called $\partial\bar{\partial}$-manifold, if $X$ satisfies the $\partial\bar{\partial}$-lemma. As in \cite{Pop19}, we shall mean by a Calabi-Yau $\partial\bar{\partial}$-manifold $X$, a $\partial\bar{\partial}$-manifold $X$ where the canonical bundle $K_X$ is trivial. Equivalently, $X$ has a nowhere vanishing holomorphic $n$-forms. Let us consider that $(X,\omega)$ is a Gauduchon Calabi-Yau $\partial\bar{\partial}$-manifold.\\

Given a holomorphic family of compact complex manifold, i.e. the following map from a complex manifold to an open ball containing the origin in $\C^m$ for some $m\in \N^*$:
$$\pi:\mathcal{X}\longrightarrow \Delta $$
is a proper holomorphic submersion, then we have $X_t=\pi^{-1}(t)$ for $t\in \Delta\setminus \{0\}$.\\
Let $(X_t)_{t\in\Delta}$ be a deformation of the complex structure of $X=X_0$. By Wu's result \cite{Wu}, the $\partial\bar{\partial}$-property is deformation open. On the other hand, It is proved, in (\cite{B19}, Conclusion 4.4), that the notion of $p$-pluriclosed (or briefly $p$-SKT if there exists a strictly weakly positive $(p,p)$-form for $p\in\{0,\cdots,n\}$ that is $\partial\bar{\partial}$-closed) compact complex $\partial\bar{\partial}$-manifold is open under small deformations. In the case where $p=n-1$, the $p$-SKT property is nothing but Gauduchon metric, which means that the Gauduchon $\partial\bar{\partial}$-manifold is also open under holomorphic deformations. Therefore $X_t$ is a Gauduchon Calabi-Yau $\partial\bar{\partial}$-manifold for every $t\in\Delta$ close enough to $0$.\\

Remark that all balanced manifolds provide examples of Gauduchon manifolds. One can consider the example of (\cite{FOU}, Theorem 5.2) pointed out in \cite{Pop19} (that is the construction of a solvmanifold of real dimension 6 with a holomorphic family of complex structures $J_a$, $a\in\Delta=\{a\in\C \, /\, \mid a \mid <1 \}$ such that $(M,J_a)$ is a balanced Calabi-Yau $\partial\bar{\partial}$-manifold for any $a\neq 0$ but not of class $\mathcal{C}$, hence not K\"ahler) as an example of Gauduchon Calabi-Yau $\partial\bar{\partial}$-manifold that is not K\"ahler.\\

Notice that the Bogomolov-Tian-Todorov theorem shows that on a K\"ahler manifolds, the base space $\Delta$ of the Kuranishi family is isomorphic to an open subset of $H^{0,1}(X,T^{1,0}X)$ (see \cite{Tia87}) (i.e. the Kuranishi family is unobstructed). Later on, this theorem has been weakened, in \cite{Pop19}, by proving that it remains true on $\partial\bar{\partial}$-manifolds, so
$$T_0\Delta \simeq H^{0,1}(X,T^{1,0}X). $$

On a balanced Calabi-Yau $\partial\bar{\partial}$-manifold, the deformations of $X_0=X$ co-polarised by the balanced class $[\omega^{n-1}]_{\bar{\partial}}\in H^{n-1,n-1}_{\bar{\partial}}(X,\C)$ (by definition, if the De Rham class $\{\omega^{n-1}\}\in H^{2n-2}_{DR}(X,\C)$ is of type $(n-1,n-1)$ for the complex structure $J_t$ of $X_t$), in \cite{Pop19}, are parametrised by:
\begin{equation}\label{balanced}
H^{0,1}(X,T^{1,0}X)_{[\omega^{n-1}]}:=\{[v]\in H^{0,1}(X,T^{1,0}X) \, /\, [v\lrcorner \omega^{n-1}]=0\in H^{n-2,n}_{\bar{\partial}}(X,\C) \}.
\end{equation}

Now, let us consider that $X$ is a compact complex $\partial\bar{\partial}$-manifold of complex dimension $n$. It is proved in \cite{Pop15}, that every Aeppli cohomology class contains a $d$-closed representative and the following map called the {\it Hodge-Aeppli decomposition}:
\begin{equation}\label{H-A iso}
\begin{array}{ll}
&H^k_{DR}(X,\C)\overset{\simeq}{\longrightarrow} \displaystyle{\bigoplus_{p+q=k}} H^{p,q}_{A}(X,\C)\\
&\left\{\displaystyle\sum_{p+q=k} \alpha^{p,q}\right\}\longmapsto \displaystyle{\sum_{p+q=k}}[\alpha^{p,q}]
\end{array}
\end{equation}

\noindent is a canonical isomorphism for any $k=0,\cdots,2n$, where 
$$H^k_{DR}(X,\C)=\dfrac{\ker\{d: C^\infty_k(X,\mathbb{C})\longrightarrow C^\infty_{k+1}(X,\mathbb{C})\}}{Im\{d: C^\infty_{k-1}(X)\longrightarrow C^\infty_{k}(X) \}}.$$
is the {\it De Rham cohomology group} of degree $k$, while
\begin{equation}\label{Aeppli cohom}
H^{p,q}_{A}(X,\C)=\dfrac{\ker\{\partial\bar{\partial}: C^\infty_{p,q}(X,\mathbb{C})\longrightarrow C^\infty_{p+1,q+1}(X,\mathbb{C})\}}{Im\{\partial: C^\infty_{p-1,q}(X)\longrightarrow C^\infty_{p,q}(X) \}+ Im\{\bar{\partial}: C^\infty_{p,q-1}(X)\longrightarrow C^\infty_{p,q}(X)\}}
\end{equation}

\noindent is the {\it Aeppli cohomology group} of type $(p,q)$ with $p+q=k$. The following $4^{\mbox{\tiny{th}}}$ order Aeppli and Bott-Chern Laplacians (cf. \cite{Sch07})
$$\Delta_{A}=\partial\partial^*+\bar{\partial}\bar{\partial}^*+\bar{\partial}^*\partial^*\partial\bar{\partial}+\partial\bar{\partial}\bar{\partial}^*\partial^*+\partial\bar{\partial}^*\bar{\partial}\partial^*+\bar{\partial}\partial^*\partial\bar{\partial}^* : C^\infty_{p,q}(X,\C)\longrightarrow C^\infty_{p,q}(X,\C),$$
$$\Delta_{BC}=\partial^*\partial+\bar{\partial}^*\bar{\partial} +\partial\bar{\partial}\bar{\partial}^*\partial^*+\bar{\partial}^*\partial^*\partial\bar{\partial}+\bar{\partial}^*\partial\partial^*\bar{\partial}+\partial^*\bar{\partial}\bar{\partial}^*\partial : C^\infty_{p,q}(X,\C)\longrightarrow C^\infty_{p,q}(X,\C).$$
are elliptic and formally self-adjoint where $\partial^*$ (resp. $\bar{\partial^*}$) is the formal adjoint of $\partial$ (resp. $\bar{\partial}$) w.r.t. the $L^2$ scalar product defined by $\omega$, and we have:
\begin{equation}\label{(1)}
\ker\Delta_A=\ker\partial^*\cap\ker\bar{\partial}^*\cap\ker\partial\bar{\partial},
\end{equation}
\begin{equation}\label{(2)}
\ker\Delta_{BC}=\ker\partial\cap\ker\bar{\partial}\cap\ker(\partial\bar{\partial})^*.
\end{equation}
Moreover, we have the following orthogonal (w.r.t. the $L^2$ scalar product defined by $\omega$) three-space decomposition
\begin{equation}\label{(1)'}
C^\infty_{p,q}(X,\C)=\ker\Delta_A\bigoplus\, (\mbox{Im }\partial+\mbox{ Im } \bar{\partial})\,\bigoplus \mbox{ Im }(\partial\bar{\partial})^*,
\end{equation}
\begin{equation}\label{(2)'}
C^\infty_{p,q}(X,\C)=\ker\Delta_{BC}\,\bigoplus \mbox{ Im }(\partial\bar{\partial}) \, \bigoplus\, (\mbox{Im }\partial^*+\mbox{ Im } \bar{\partial}^*),
\end{equation} 
and the Hodge isomorphism
$$ H^{p,q}_A(X,\C)\simeq \ker \Delta_A \hspace{1cm} \mbox{ and } \hspace{1cm} H^{p,q}_{BC}(X,\C)\simeq \ker \Delta_{BC}.$$
Besides, $\omega$ induces the Hodge star isomorphism $ *_{\omega}=*:C^\infty_{p,q}(X,\C)\longrightarrow C^\infty_{n-q,n-p}(X,\C) $. It is well known that
\begin{equation}\label{(3)}
\partial^*=-*\bar{\partial}*, \hspace*{1.5ex} \bar{\partial}^*=-*\partial*, \hspace*{1.5ex} \mbox{ and }  \hspace*{1.5ex}   (\partial\bar{\partial})^*=-*\partial\bar{\partial}*.
\end{equation}
In consequence, $\ker\Delta_A$ and $\ker\Delta_{BC}$ are related as follows

$$
\begin{array}{ll}
\alpha\in\ker\Delta_A &\Longleftrightarrow \alpha\in\ker\partial\bar{\partial}, \alpha\in\ker\partial^* \mbox{ and } \alpha\in\ker\bar{\partial}^*\\
&\Longleftrightarrow *\alpha\in\ker(\partial\bar{\partial})^*, *\alpha\in\ker\partial \mbox{ and } *\alpha\in\ker\bar{\partial}
\end{array}
$$
\begin{equation}\label{(5)}
\Longleftrightarrow  *\alpha\in\ker\Delta_{BC}. \hspace{2.5cm}
\end{equation}

Taking our cue from the concept of balanced co-polarised deformations and following Popovici's method in \cite{Pop19}, we shall define the weakened notion of the balanced co-polarisation to Gauduchon co-polarisation and prove the following main result of Section \ref{copolar}:
\begin{Prop}
Let $X$ be a Gauduchon Calabi-Yau $\partial\bar{\partial}$-manifold with $\dim_{\C}X=n$ and let $\omega$ be an arbitrary Hermitian metric on $X$. For any Aeppli-Gauduchon class $[\omega^{n-1}]_A\in H^{n-1,n-1}_A(X,\C)$, let $\omega^{n-1}$  be the $\omega$-minimal $d$-closed representative of the class $[\omega^{n-1}]_A$. Then, the following vector subspace of $H^{0,1}(X,T^{1,0}X)$
$$H^{0,1}(X,T^{1,0}X)_{[\omega^{n-1}]_A}:=\{[v]\in H^{0,1}(X,T^{1,0}X)\, /\, [v\lrcorner \omega^{n-1}]_A=0\in H^{n-2,n}_A(X,\C) \}. $$
is well defined, i.e. the Aeppli class $[v\lrcorner \omega^{n-1}]_A\in H^{n-2,n}_A(X,\C)$ is independent of the choice of representative $v$ in the class $[v]\in H^{0,1}(X,T^{1,0}X)$. Moreover, we have
$$ T_0\Delta_{[\omega^{n-1}]_A}\simeq H^{0,1}(X,T^{1,0}X)_{[\omega^{n-1}]_A}. $$
\end{Prop} 
\vspace{2ex}

More precisely (this fact hinted at in the Introduction to \cite{Pop15}), suppose that $X$ is a Calabi-Yau $\partial\bar{\partial}$-manifold. Having fixed a Gauduchon metric $\omega$ on $X$, consider the small deformations $X_t$ that are co-polarised by the Aeppli class $[\omega^{n-1}]_{A}\in H^{n-1,n-1}_{A}(X,\C) $ (or by, thanks to the $\partial\bar{\partial}$-assumption on $X$, the De Rham class of degree $2n-2$ that is its canonical image in $H^{2n-2}_{DR}(X,\C)$). We say that a small deformation $X_t$ of $X$ is {\bf co-polarised by the Aeppli-Gauduchon class} (or {\bf De Rham-Gauduchon class} as called in \cite{Pop15}) $ [\omega^{n-1}]_{A}\in H^{n-1,n-1}_{A}(X,\C)$ if the De Rham class  $\{\omega^{n-1}\}\in H^{2n-2}_{DR}(X,\C)$ is of type $(n-1,n-1)$ for the complex structure of $X_t$.\\
It is clear that, as in \cite{Pop15}, if the manifold $X$ satisfies the $\partial\bar{\partial}$-property, then (for any $p,q=0,\cdots,n$) the Aeppli cohomology group $H^{p,q}_{A}(X,\C)$ is injecting canonically into the De Rham cohomology $H^{p+q}_{DR}(X,\C)$. Moreover, since the De Rham cohomology does not depend on the complex structure, while the Aeppli cohomology groups of the fibres $X_t$ depend on $t\in \Delta$ and the $\partial\bar{\partial}$-property is deformation-open, then one obtains the Hodge-Aeppli decomposition (\ref{H-A iso}) of degree $2n-2$ of $X_t$:
\begin{equation}\label{Hodge-Aeppli decomposition}
H^{2n-2}_{DR}(X,\C)\simeq H^{n,n-2}_{A}(X_t,\C)\bigoplus H^{n-1,n-1}_{A}(X_t,\C)\bigoplus H^{n-2,n}_{A}(X_t,\C)
\end{equation}
and the Hodge symmetry $ H^{n-2,n}_{A}(X_t,\C)\simeq\overline{ H^{n,n-2}_{A}(X_t,\C)}$ for $t$ in a possibly shrunk $\Delta$.\\
It is worth mentioning that the initial De Rham class that is in $H^{2n-2}_{DR}(X,\C)$ is of type $(n-1,n-1)$ for the complex structure $J_t$ of $X_t$, by definition, if and only if it is in $H^{n-1,n-1}_{A}(X,\C)$. In other words, if and only if the projections of the class $\{\omega^{n-1,n-1}\}\in H^{2n-2}_{DR}(X,\C) $ on  $H^{n,n-2}_{A}(X,\C)$ and  $ H^{n-2,n}_{A}(X,\C)$ defined by the Hodge-Aeppli decomposition vanish. Furthermore, the fact that the De Rham class being real is equivalent to only one of these projections, for example that on $ H^{n-2,n}_{A}(X,\C)$, vanishes. \\

Recall that a $(p,q)$-form $\alpha$ on $X$ is said to be primitive for $\omega$ if it satisfies the condition $\omega^{n-k+1}\wedge \alpha=0$ with $k=p+q\leq n$. Let us denote the set of primitive $(p,q)$-forms by $C^\infty_{p,q}(X,\C)_{prim}$. \\

\noindent Suppose that $\alpha \in C^\infty_{p,q}(X,\C)_{prim}$ is a primitive $(p,q)$-form, by (\cite{Voi02}, Proposition 6.29, p. 150), we have:
\begin{equation}\label{prim form}
*\alpha=(-1)^{(p+q)(p+q+1)/2}\,\, i^{p-q}\,\, \dfrac{\omega^{n-p-q}\wedge\alpha}{(n-p-q)!}.
\end{equation}

\noindent For $p=n-1$ and $q=1$, the above formula yields
\begin{equation}\label{(4)}
*\alpha=i^{n^2+2n-2}\,\,\alpha.
\end{equation}

Given a compact Gauduchon Calabi-Yau $\partial\bar{\partial}$-manifold, we define in Subsection \ref{sect prim} (cf. Definition \ref{def Gprim}) {\bf the space of primitive $(n-1,1)$-class} as:
$$ [v\lrcorner u]\in H^{n-1,1}_{Gprim}(X,\C) \hspace*{1.5ex} \Longleftrightarrow \hspace*{1.5ex} [v\lrcorner \omega^{n-1}]_A=0\in H^{n-2,n}_A(X,\C). $$
for any $[v]\in H^{0,1}(X,T^{1,0}X)$, where $u\neq 0$ is a holomorphic $(n,0)$-form and $\omega^{n-1}$  is the $\omega$-minimal $d$-closed representative (see Definition \ref{refinement}) of the class $[\omega^{n-1}]_A\in H^{n-1,n-1}_A(X,\C)$. Thus, we have

$$ T_0\Delta_{[\omega^{n-1}]_A}\simeq H^{0,1}(X,T^{1,0}X)_{[\omega^{n-1}]_A}\simeq H^{n-1,1}_{Gprim}(X,\C).$$

If $\omega$ is balanced, then the space of primitive $(n-1,1)$-class $H^{n-1,1}_{Gprim}(X,\C)$ coincides with the space of primitive $(n-1,1)$-class  $H^{n-1,1}_{prim}(X,\C)$ in the balanced case. In the rest of this Subsection, we will discuss the existence of a primitive $d$-closed representative of a primitive $(n-1,1)$-class. If such representative exists, it will play an important role in the discussion of Weil-Petersson metrics in Section \ref{section prim+WP} (cf. Subsection \ref{WP metrics}).

 Afterward, in Subsection \ref{WP metrics}, we define the Weil-Petersson metrics on $\Delta_{[\omega^{n-1}]_A}$ and compare the smooth positive definite $(1,1)$-forms $\tilde{\omega}^{(2)}_{WP}$ associated with the Weil-Petersson metric $g^{(2)}_{WP}$ (cf. Definition \ref{def of WP metrics}) on $\Delta_{[\omega^{n-1}]_A}$ with the period map metric $\gamma$ defined and discussed in \cite{Pop19}.\\

In the last section, we will establish the interaction between various kind of forms, especially the $p$-SKT on a compact complex  $h$-$\partial\bar{\partial}$-manifold of complex dimension $n$.\\
Let $X$ be a compact complex $n$-dimensional manifold. Notice that, in \cite{Pop17}, the adiabatic limit construction of the differential operator $d_h=h\partial+\bar{\partial}$ has been introduced for every constant $h>0$ with the $d$-cohomology and $d_h$-cohomology are related by the following isomorphism

\begin{equation}\label{DRdh}
\begin{array}{ll}
& H^{k}_{DR}(X,\C) \longrightarrow H^{k}_{d_h}(X,\C)\\
&\hspace{1cm} \{\alpha\}_d \longmapsto \{\theta_h\alpha\}_{d_h}
\end{array}
\end{equation} 

\noindent where $\theta_h$ is defined by the pointwise isomorphism
$$
\begin{array}{ll}
& \Lambda^{p,q} T^*X \longrightarrow \Lambda^{p,q} T^*X\\
&\hspace{1.2cm} \alpha \longmapsto \theta_h\alpha=h^p\alpha,
\end{array}
$$
and the $d_h$-cohomology is defined by:
 $$H^{k}_{d_h}(X,\C)=\ker\{d_h:C^\infty_{k}(X,\C)\longrightarrow C^\infty_{k+1}(X,\C)\}\diagup\mbox{ Im }\{d_h:C^\infty_{k-1}(X,\C)\longrightarrow C^\infty_{k}(X,\C)\}.$$\\
 On the other hand, we shall mean by $h$-$\partial\bar{\partial}$-manifold for any $h\in\R\setminus\{0\}$, a manifold that satisfies the $h$-$\partial\bar{\partial}$-lemma (see definition \ref{recall defn}) as defined in \cite{BP18}. As mentioned above, it is already proved in \cite{B19} that the $p$-SKT property is open under holomorphic deformations on a $\partial\bar{\partial}$-manifold. The purpose of this section is the following
\begin{The}
The $p$-SKT $h$-$\partial\bar{\partial}$-property is open under holomorphic deformations of the complex structure.
\end{The}
To prove this result, we will construct a new form $\Omega$ that will be called {\bf $hp$-Hermitian symplectic ($hp$-HS)}  if $\Omega$ is the $(p,p)$-type component of a real $d_h$-closed $(2p)$-form (see definition \ref{hpHS defn}). We observe in theorem \ref{hp-HS p-SKT} that the properties $hp$-HS and $p$-SKT are equivalent on any compact complex $h$-$\partial\bar{\partial}$-manifold. In the end, proving the openness under holomorphic deformations of the $hp$-HS $h$-$\partial\bar{\partial}$ property allows us to prove that $p$-SKT $h$-$\partial\bar{\partial}$-property is deformation-open.
\section{Holomorphic deformations of the co-polarised Gauduchon Calabi-Yau $\partial\bar{\partial}$-manifolds}\label{copolar}

Suppose that $X$ is a compact Gauduchon Calabi-Yau $\partial\bar{\partial}$-manifold of complex dimension $n$. Let $\pi: \mathcal{X}\longrightarrow \Delta$ be a proper holomorphic submersion between complex manifolds. The fibres $X_t:=\pi^{-1}(t)$ for $t\in\Delta$ are compact complex manifolds of the same dimension $n$ and are $C^\infty$ diffeomorphic to a fixed $C^\infty$ manifold $X$. Since $X$ is a Calabi-Yau $\partial\bar{\partial}$-manifold, the base space of the Kuranishi family is isomorphic to an open subset of $H^{0,1}(X,T^{1,0}X)$  (cf. \cite{Pop19}, Theorem 1.2).\\

Besides Wu's openness under holomorphic deformations of the $\partial\bar{\partial}$-property, the Gauduchon $\partial\bar{\partial}$-manifold is also deformation-open as mentioned in the Introduction. Moreover, from the deformation openness of the triviality of the canonical bundle $K_{X_t}$ of $X_t$ under the $\partial\bar{\partial}$-assumption, one can conclude that small deformations of Gauduchon Calabi-Yau  $\partial\bar{\partial}$-manifolds are again Gauduchon Calabi-Yau  $\partial\bar{\partial}$-manifolds.\\

Recall that the concept of local deformations $X_t$ of $X$ that are co-polarised by the balanced class $[\omega^{n-1}]\in H^{n-1,n-1}_{\bar{\partial}}(X,\C)\subset H^{2n-2}_{DR}(X,\C)$
 has been introduced, in \cite{Pop19}, by requiring that the De Rham cohomology class $\{\omega^{n-1}\}\in H^{2n-2}_{DR}(X,\C)$ be of type $(n-1,n-1)$ for the complex structure $J_t$ of $X_t$ and are parametrised by (\ref{balanced}). Following the pattern of the deformations of balanced Calabi-Yau $\partial\bar{\partial}$-manifold, we shall propose the following weakened notion:

\begin{Def}
Fixing an Aeppli-Gauduchon class 
$$[\omega^{n-1}]_A\in H^{n-1,n-1}_A(X,\C)\subset H^{2n-2}_{DR}(X,\C), $$
the fibre $X_t$ is said to be {\bf{co-polarised by}} $[\omega^{n-1}]_A$ if the De Rham class $\{\omega^{n-1}\}\in H^{2n-2}_{DR}(X,\C)$ is of type $(n-1,n-1)$ for the complex structure of $X_t$.
\end{Def}

Denote by $\Delta_{[\omega^{n-1}]_A}\subset\Delta$ {\bf the open subset of local deformations of $X$ co-polarised by the Aeppli-Gauduchon class} $[\omega^{n-1}]_A\in H^{n-1,n-1}_A(X,\C)$ and 
$$\pi:\mathcal{X}_{[\omega^{n-1}]_A}\longrightarrow \Delta_{[\omega^{n-1}]_A} $$
is {\bf the local universal family of co-polarised deformations of} $X$.\\

Since the fibre $X_t$ is $\partial\bar{\partial}$-manifold for any $t\in \Delta$ close to 0, the Hodge-Aeppli decomposition (\ref{Hodge-Aeppli decomposition}) holds on $X_t$. Let $\{\omega^{n-1}\}_{DR}\in H^{2n-2}_{DR}(X,\C)$ be the co-polarising De Rham cohomology class, then the splitting of $\omega^{n-1}$ yields
$$\omega^{n-1}=\alpha^{n-2,n}_t+\alpha^{n-1,n-1}_t+\alpha^{n,n-2}_t. $$
Notice that the family $(\alpha^{n-1,n-1}_t)_{t\in\Delta}$ vary smoothly with $t$ and $\alpha^{n-1,n-1}_0=\omega^{n-1}>0$, so $\alpha^{n-1,n-1}_t$ is a positive definite real $(n-1,n-1)$-form on $X_t$ for any $t$ close sufficiently to $0$. Thus, by \cite{Mic82}, there exists a unique smooth $(1,1)$-form $\tilde{\omega}_t>0$ on $X_t$ such that $\alpha^{n-1,n-1}_t=\tilde{\omega}^{n-1}_t$.\\
 The $(2n-2)$-form $\omega^{n-1}$ being $d$-closed implies that $\tilde{\omega}^{n-1}$ is $\partial_t\bar{\partial}_t$-closed and that $[\tilde{\omega}^{n-1}_t]_A\in H^{n-1,n-1}_A(X_t,\C)$ with 
 $$
 \begin{array}{ll}
 Pr: & H^{2n-2}_{DR}(X,\C) \longrightarrow H^{n-1,n-1}_A(X_t,\C)\\
 &\hspace{0.5cm}\{\omega^{n-1}\}_{DR}\longmapsto [\tilde{\omega}^{n-1}_t]_A
 \end{array}
 $$
is the projection defined by the Hodge-Aeppli decomposition (\ref{Hodge-Aeppli decomposition}). Furthermore, if we assume that the class $\{\omega^{n-1}\}_{DR}\in H^{2n-2}_{DR}(X,\C) $ is of type $(n-1,n-1)$ for the complex structure of $X_t$, then  $\{\omega^{n-1}\}_{DR}=\{\tilde{\omega}_t^{n-1}\}_{DR}$ for every $t\in\Delta$ close enough to $0$. Therefore, for every $t\in \Delta$ close to $0$, the class $\{\omega^{n-1}\}_{DR}$ contains the Gauduchon metric $\tilde{\omega}_t^{n-1}$ for the complex structure $J_t$ of $X_t$.\\
 
 The case where $\partial\phi$ (of the following lemma) is replaced by $\omega^{n-1}$ has been already proved in (\cite{Pop19}, Lemma 4.3) and it is mentioned that this result still valid for forms of any type. For the reader's convenience, we will give, explicitly, the proof of the following important tool to prove the Proposition \ref{Prop}.
\begin{Lem}\label{lem}
Suppose that $X$ is a compact complex $n$-dimensional manifold. Then for any $(p,q)$-form $\phi\in C^\infty_{p,q}(X,\C)$ with $p,q=0,\cdots,n$, we have:
\begin{itemize}
\item[(a)] $\bar{\partial}(\zeta\lrcorner \partial \phi)=\bar{\partial}\zeta\lrcorner \phi-\zeta\lrcorner\bar{\partial}\partial \phi$, \hspace{2ex} for every $\zeta\in C^\infty(X,T^{1,0}X)$;
\item[(b)] $\bar{\partial}(v\lrcorner \partial \phi)=\bar{\partial}v\lrcorner \phi+v\lrcorner\bar{\partial}\partial \phi$, \hspace{2ex} for every $v\in C^\infty_{0,1}(X,T^{1,0}X)$.
\end{itemize}
\end{Lem}
\begin{proof}
To prove the pointwise indentities above, we fix an arbitrary point $x\in X$ and choose local holomorphic coordinates $z_1,\cdots,z_n$ about $x$ such that
$$\phi=\displaystyle{\sum_{ \overset{\mid I\mid =p}{\mid J\mid =q}}} \,\phi_{IJ}\, dz_I\wedge d\overline{z}_J.   $$
\begin{itemize}
\item[(a)] Let us start by computing $\zeta\lrcorner \partial \phi$. Notice that
$$\partial \phi=\displaystyle{\sum_{ \overset{\mid I\mid =p}{\mid J\mid =q}}} \displaystyle{\sum_{k=1}^n}\,\dfrac{\partial \phi_{IJ}}{\partial z_k} \, dz_k\wedge dz_I\wedge d\overline{z}_J \hspace*{4ex}  \mbox{ and } \hspace*{4ex}  \zeta=\displaystyle{\sum_{j=1}^n} \,\zeta_j\, \dfrac{\partial}{\partial z_j}$$
$$
\begin{array}{ll}
\zeta\lrcorner \partial \phi &= \displaystyle{\sum_{j=1}^n} \,\zeta_j\, \dfrac{\partial}{\partial z_j} \lrcorner \displaystyle{\sum_{ I,J}} \displaystyle{\sum_{k=1}^n}\,\dfrac{\partial \phi_{IJ}}{\partial z_k} \, dz_k\wedge dz_I\wedge d\overline{z}_J \\
&=\displaystyle{\sum_{I,J}} \displaystyle{\sum_{j,k}}\,\zeta_j \dfrac{\partial \phi_{IJ}}{\partial z_k}  \dfrac{\partial}{\partial z_j} \lrcorner(dz_k\wedge dz_I\wedge d\overline{z}_J ).
\end{array}
$$
Putting $I=(i_1,\cdots,i_p)$, a direct calculation of $ \dfrac{\partial}{\partial z_j} \lrcorner(dz_k\wedge dz_I\wedge d\overline{z}_J )$ yields
\begin{equation}\label{contrac}
 \dfrac{\partial}{\partial z_j} \lrcorner(dz_k\wedge dz_I\wedge d\overline{z}_J )=\delta_{jk}\,dz_I\wedge d\overline{z}_J-\sum_l \delta_{jl} (-1)^{l-1} dz_k\wedge dz_{I\setminus{\{j=i_l\}}}\wedge d\overline{z}_J
\end{equation}
where $dz_{I\setminus{\{i_l\}}}=dz_{i_1}\wedge\cdots\wedge \widehat{dz_{i_l}}\wedge\cdots\wedge dz_p$ while the symbol $\widehat{dz_{i_l}}$ means that the term $dz_{i_l}$ is omitted. Thus, we obtain
$$
\begin{array}{ll}
\zeta\lrcorner \partial \phi &= \displaystyle{\sum_{I,J}} \displaystyle{\sum_{k}}\,\zeta_k \dfrac{\partial \phi_{IJ}}{\partial z_k}\, dz_I\wedge d\overline{z}_J -\displaystyle{\sum_{I,J}} \displaystyle{\sum_{j,k,l}}\delta_{jl} (-1)^{l-1}\zeta_j \dfrac{\partial \phi_{IJ}}{\partial z_k}  dz_k\wedge dz_{I\setminus{\{j=i_l\}}}\wedge d\overline{z}_J\\
&=\displaystyle{\sum_{I,J}} \displaystyle{\sum_{k}}\,\zeta_k \dfrac{\partial \phi_{IJ}}{\partial z_k}\, dz_I\wedge d\overline{z}_J -\displaystyle{\sum_{I,J}} \displaystyle{\sum_{j,k}} (-1)^{j-1}\zeta_j \dfrac{\partial \phi_{IJ}}{\partial z_k}  dz_k\wedge dz_{I\setminus{\{i_j\}}}\wedge d\overline{z}_J.
\end{array}
$$

Thereby
\begin{equation}\label{1}
\bar{\partial}( \zeta\lrcorner \partial \phi)=\displaystyle{\sum_{I,J}} \displaystyle{\sum_{k,r}}\,\frac{\partial\left(\zeta_k \dfrac{\partial \phi_{IJ}}{\partial z_k}\right)}{\partial \overline{z}_r}\, d\overline{z}_r \wedge dz_I\wedge d\overline{z}_J -\displaystyle{\sum_{I,J}} \displaystyle{\sum_{j,k,r}} (-1)^{j-1}\dfrac{\partial\left(\zeta_j \dfrac{\partial \phi_{IJ}}{\partial z_k}\right)}{\partial \overline{z}_r} d\overline{z}_r\wedge dz_k\wedge dz_{I\setminus{\{i_j\}}}\wedge d\overline{z}_J.
\end{equation}

Since $\bar{\partial}\zeta=\displaystyle{\sum_{j,r}} \,\dfrac{\partial\zeta_j}{\partial\overline{z}_r}\, d\overline{z}_r\wedge\dfrac{\partial}{\partial z_j}$, then
\begin{equation}\label{2}
\begin{array}{ll}
\bar{\partial}\zeta\lrcorner\partial \phi &=\displaystyle{\sum_{j,r}} \,\dfrac{\partial\zeta_j}{\partial\overline{z}_r}\, d\overline{z}_r\wedge\dfrac{\partial}{\partial z_j}\lrcorner\displaystyle{\sum_{ I,J}} \displaystyle{\sum_{k}}\,\dfrac{\partial \phi_{IJ}}{\partial z_k} \, dz_k\wedge dz_I\wedge d\overline{z}_J \\
&=\displaystyle{\sum_{ I,J}} \displaystyle{\sum_{j,k,r}}\,\dfrac{\partial\zeta_j}{\partial\overline{z}_r}\dfrac{\partial \phi_{IJ}}{\partial z_k} \, d\overline{z}_r\wedge\dfrac{\partial}{\partial z_j}\lrcorner (dz_k\wedge dz_I\wedge d\overline{z}_J)\\
&\overset{(i)}{=} \displaystyle{\sum_{ I,J}} \displaystyle{\sum_{k,r}}\,\dfrac{\partial\zeta_k}{\partial\overline{z}_r}\dfrac{\partial \phi_{IJ}}{\partial z_k} \, d\overline{z}_r\wedge dz_I\wedge d\overline{z}_J-\displaystyle{\sum_{I,J}} \displaystyle{\sum_{j,k,r}} (-1)^{j-1}\dfrac{\partial\zeta_j}{\partial \overline{z}_r}\dfrac{\partial \phi_{IJ}}{\partial z_k} d\overline{z}_r\wedge dz_k\wedge dz_{I\setminus{\{i_j\}}}\wedge d\overline{z}_J,
\end{array}
\end{equation}
where the equality $(i)$ holds from (\ref{contrac}).\\

On the other side, we have $\bar{\partial}\partial \phi=\displaystyle{\sum_{ I,J}} \displaystyle{\sum_{k,r}}\,\dfrac{\partial^2 \phi_{IJ}}{\partial\overline{z}_r\partial z_k} \, d\overline{z}_r\wedge dz_k\wedge dz_I\wedge d\overline{z}_J $, thus
\begin{equation}\label{3}
\begin{array}{ll}
\zeta\lrcorner \bar{\partial}\partial \phi &=\displaystyle{\sum_{j}} \,\zeta_j\, \dfrac{\partial}{\partial z_j}\lrcorner \displaystyle{\sum_{ I,J}} \displaystyle{\sum_{k,r}}\,\dfrac{\partial^2 \phi_{IJ}}{\partial\overline{z}_r\partial z_k} \, d\overline{z}_r\wedge dz_k\wedge dz_I\wedge d\overline{z}_J\\
&= \displaystyle{\sum_{ I,J}} \displaystyle{\sum_{j,k,r}}\,\zeta_j\dfrac{\partial^2 \phi_{IJ}}{\partial\overline{z}_r\partial z_k} \, \dfrac{\partial}{\partial z_j}\lrcorner ( d\overline{z}_r\wedge dz_k\wedge dz_I\wedge d\overline{z}_J)\\
&=-\displaystyle{\sum_{ I,J}} \displaystyle{\sum_{j,k,r}}\,\zeta_j\dfrac{\partial^2 \phi_{IJ}}{\partial\overline{z}_r\partial z_k} \,  d\overline{z}_r\wedge \dfrac{\partial}{\partial z_j}\lrcorner (dz_k\wedge dz_I\wedge d\overline{z}_J)\\
&\overset{(ii)}{=} -\displaystyle{\sum_{ I,J}} \displaystyle{\sum_{k,r}}\,\zeta_k\dfrac{\partial^2 \phi_{IJ}}{\partial\overline{z}_r\partial z_k} \, d\overline{z}_r\wedge dz_I\wedge d\overline{z}_J+\displaystyle{\sum_{I,J}} \displaystyle{\sum_{j,k,r}} (-1)^{j-1}\zeta_j\dfrac{\partial^2 \phi_{IJ}}{\partial\overline{z}_r\partial z_k} d\overline{z}_r\wedge dz_k\wedge dz_{I\setminus{\{i_j\}}}\wedge d\overline{z}_J,
\end{array}
\end{equation}
where the equality $(ii)$ holds from (\ref{contrac}).\\

From the computations of (\ref{1}), (\ref{2}) and (\ref{3}), the assertion (a) holds.
\item[(b)] Now, suppose that $v=\displaystyle{\sum_{l=1}^n} \,v_l\, d\overline{z}_l\wedge \dfrac{\partial}{\partial z_l}$. In the same way as in $(a)$, we have the following computations
$$\bar{\partial}(v\lrcorner \partial \phi)= $$
\begin{equation}\label{1'}
\displaystyle{\sum_{I,J}} \displaystyle{\sum_{k,r}}\,\frac{\partial\left(v_k \dfrac{\partial \phi_{IJ}}{\partial z_k}\right)}{\partial \overline{z}_r}\, d\overline{z}_r \wedge d\overline{z}_k\wedge dz_I\wedge d\overline{z}_J -
\displaystyle{\sum_{I,J}} \displaystyle{\sum_{l,k,r}} (-1)^{l-1}\dfrac{\partial\left(v_l \dfrac{\partial \phi_{IJ}}{\partial z_k}\right)}{\partial \overline{z}_r} d\overline{z}_r\wedge d\overline{z}_l\wedge dz_k\wedge dz_{I\setminus{\{i_l\}}}\wedge d\overline{z}_J,
\end{equation}
\begin{equation}\label{2'}
\bar{\partial}v\lrcorner \partial \phi=\displaystyle{\sum_{ I,J}} \displaystyle{\sum_{k,r}}\,\dfrac{\partial v_k}{\partial\overline{z}_r}\dfrac{\partial \phi_{IJ}}{\partial z_k} \, d\overline{z}_r\wedge d\overline{z}_k\wedge dz_I\wedge d\overline{z}_J-\displaystyle{\sum_{I,J}} \displaystyle{\sum_{k,l,r}} (-1)^{l-1}\dfrac{\partial v_l}{\partial \overline{z}_r}\dfrac{\partial \phi_{IJ}}{\partial z_k} d\overline{z}_r\wedge d\overline{z}_l\wedge dz_k\wedge dz_{I\setminus{\{i_l\}}}\wedge d\overline{z}_J,
\end{equation}
\begin{equation}\label{3'}
v\lrcorner \bar{\partial}\partial \phi=\displaystyle{\sum_{ I,J}} \displaystyle{\sum_{k,r}}\,v_k\dfrac{\partial^2 \phi_{IJ}}{\partial\overline{z}_r\partial z_k} \, d\overline{z}_r\wedge d\overline{z}_k\wedge dz_I\wedge d\overline{z}_J-\displaystyle{\sum_{I,J}} \displaystyle{\sum_{k,l,r}} (-1)^{l-1}v_l\dfrac{\partial^2 \phi_{IJ}}{\partial\overline{z}_r\partial z_k} d\overline{z}_r\wedge d\overline{z}_l\wedge dz_k\wedge dz_{I\setminus{\{i_l\}}}\wedge d\overline{z}_J.
\end{equation}
In the end, taking the sum of the equations (\ref{2'}) and (\ref{3'}), we obtain (\ref{1'}). This proves the assertion $(b)$ .
\end{itemize}
\end{proof}
It is well known that every Aeppli cohomology class contains a $d$-closed representative (cf. \cite{Pop15}). This means that, for every Aeppli class $[\chi]_A\in H^{p,q}_A(X,\C)$, we have 
\begin{equation}\label{closed repres}
d(\chi+\partial\phi+\bar{\partial}\psi)=0,
\end{equation}
where $\phi\in C^\infty_{p-1,q}(X,\C)$ and $\psi\in C^\infty_{p,q-1}(X,\C)$. The equation (\ref{closed repres}) is equivalent to
\begin{equation}\label{equations aeppli}
\bar{\partial}\chi=\partial\bar{\partial}\phi \hspace*{2ex} \mbox{ and } \hspace*{2ex} \partial\chi=-\partial\bar{\partial}\psi.
\end{equation} 
Notice that the solutions of (\ref{equations aeppli}) are unique up to $\ker\partial\bar{\partial}$. Then, if $\phi_{min}\in C^\infty_{p-1,q}(X,\C)$ and $\psi_{min}\in C^\infty_{p,q-1}(X,\C)$ are the minimal $L^2$-norm solutions of (\ref{equations aeppli}), then $\phi_{min}$, $\psi_{min}\in \ker(\partial\bar{\partial})^{\perp}=\mbox{ Im }(\partial\bar{\partial})^*$.
\begin{Def}\label{refinement}
Given a compact $\partial\bar{\partial}$-manifold $X$ and an arbitrary Hermitian metric $\omega$ on $X$. For any Aeppli cohomology class $[\tilde{\chi}]_A\in H^{p,q}_A(X,\C)$, suppose that $\chi$ is the  $\Delta_{A}$-harmonic representative of $[\tilde{\chi}]_A$. Let $\phi_{min}\in \mbox{ Im } (\partial\bar{\partial})^*\subset C^\infty_{p-1,q}(X,\C)$ and $\psi_{min}\in \mbox{ Im } (\partial\bar{\partial})^*\subset C^\infty_{p,q-1}(X,\C)$ be the minimal $L^2$-norm (w.r.t. $\omega$) solutions of (\ref{equations aeppli}).\\
We call the {\bf $ \omega$-minimal $d$-closed representative of the class} $[\chi]_A$, the following $d$-closed $(p,q)$-form
$$ \chi_{min}:=\chi+\partial\phi_{min}+\bar{\partial}\psi_{min}.$$
\end{Def}
\noindent  If $(X,\omega)$ is a K\"ahler manifold, then $\partial\chi=0$ and $\bar{\partial}\chi=0$. So, $\phi_{min}=0$ and $\psi_{min}=0$. Hence $\chi_{min}=\chi$.\\

We can now state the main result of this section.
\begin{Prop}\label{Prop}
Let $X$ be a Gauduchon Calabi-Yau $\partial\bar{\partial}$-manifold with $\dim_{\C}X=n$ and let $\omega$ be an arbitrary Hermitian metric on $X$. For any Aeppli-Gauduchon class $[\omega^{n-1}]_A\in H^{n-1,n-1}_A(X,\C)$, let $\omega^{n-1}$  be the $\omega$-minimal $d$-closed representative of the class $[\omega^{n-1}]_A$. Then, the following vector subspace of $H^{0,1}(X,T^{1,0}X)$
$$H^{0,1}(X,T^{1,0}X)_{[\omega^{n-1}]_A}:=\{[v]\in H^{0,1}(X,T^{1,0}X)\, /\, [v\lrcorner \omega^{n-1}]_A=0\in H^{n-2,n}_A(X,\C) \}. $$
is well defined, i.e. the Aeppli class $[v\lrcorner \omega^{n-1}]_A\in H^{n-2,n}_A(X,\C)$ is independent of the choice of representative $v$ in the class $[v]\in H^{0,1}(X,T^{1,0}X)$. Moreover, we have
$$ T_0\Delta_{[\omega^{n-1}]_A}\simeq H^{0,1}(X,T^{1,0}X)_{[\omega^{n-1}]_A}. $$
\end{Prop}

\begin{proof}

\noindent let $\omega^{n-1}$  be the $\omega$-minimal $d$-closed representative of the class $[\omega^{n-1}]_A$ and suppose that $v+\bar{\partial}\zeta$ is another representative of the class $[v]$ for some vector field $\zeta\in C^\infty(X,T^{1,0}X)$, then
$$(v+\bar{\partial}\zeta)\lrcorner\, \omega^{n-1}=v \lrcorner \,\omega^{n-1}+\bar{\partial}\zeta \,\lrcorner \omega^{n-1}.$$
We are thus reduced to showing that $\bar{\partial}\zeta \lrcorner\,\omega^{n-1} \in \mbox{ Im }\partial+\mbox{ Im }\bar{\partial}. $ Indeed:\\
 In (\cite{Pop19}, Lemma 4.3), it is already proved that:
$$
\bar{\partial}\zeta\lrcorner\,\omega^{n-1}=\bar{\partial}(\zeta\lrcorner\,\omega^{n-1})+\zeta\lrcorner\bar{\partial}\omega^{n-1}.
$$

\noindent While, by assumption, $\omega^{n-1}$ is $d$-closed. Accordingly, $\zeta\lrcorner\,\bar{\partial}\omega^{n-1}=0$. So
$$ \bar{\partial}\zeta\lrcorner\,\omega^{n-1}=\bar{\partial}(\zeta\lrcorner\,\omega^{n-1}).$$
As \, $\bar{\partial}\zeta\lrcorner \,\omega^{n-1}\in \mbox{ Im }\bar{\partial}\subset \mbox{ Im }\partial+\mbox{ Im }\bar{\partial}$, it follows that
$$[\bar{\partial}\zeta\lrcorner\,\omega^{n-1}]_A=0.$$ This proves that:
$$ [(v+\bar{\partial}\zeta)\lrcorner\,\omega^{n-1}]_A=[v \lrcorner \,\omega^{n-1}]_A.$$

\noindent Under the Kodaira-Spencer map, one can suppose that $[v]\in H^{0,1}(X,T^{1,0}X)$ is the image of $\dfrac{\partial}{\partial t_i}_{\mid t_i=0}$ 
$$
\begin{array}{ll}
\rho:& \hspace*{0.5cm} T_0\Delta \overset{\simeq}{\longrightarrow} H^{0,1}(X,T^{1,0}X)\\
& \dfrac{\partial}{\partial t_i}_{\mid t_i=0} \longmapsto [v],
\end{array}
$$
where $t_1,\cdots, t_N$ ($N=\dim_\C  H^{0,1}(X,T^{1,0}X)$) are local holomorphic coordinates about $0$ in $\Delta$ and $t=(t_1,\cdots, t_N)\in\Delta$. \\

By assumption $X_0=X$ is a $\partial\bar{\partial}$-manifold. So, by \cite{Wu}, $X_t$ is also a $\partial\bar{\partial}$-manifold for $t$ close enough to $0$ with $t\in \Delta$, then the following Hodge-Aeppli decomposition of $X_t$
$$H^{2n-2}_{DR}(X,\C)\simeq H^{n,n-2}_{A}(X_t,\C)\bigoplus H^{n-1,n-1}_{A}(X_t,\C)\bigoplus H^{n-2,n}_{A}(X_t,\C) $$
holds with the Hodge symmetry $ H^{n-2,n}_{A}(X_t,\C)\simeq\overline{ H^{n,n-2}_{A}(X_t,\C)}$. Hence, we have the following splitting of the De Rham class $\{\omega^{n-1}\}$:
$$ \{\omega^{n-1}\}=[\omega^{n-1}]_{A,t}^{n-2,n}+[\omega^{n-1}]_{A,t}^{n-1,n-1}+[\omega^{n-1}]_{A,t}^{n,n-2}$$
where $[\omega^{n-1}]_{A,t}^{n-2,n}=\overline{[\omega^{n-1}]_{A,t}^{n,n-2}}$ and $[\omega^{n-1}]_{A,t}^{n-1,n-1}$ is real. Therefore, we have:
$$\Delta_{[\omega^{n-1}]_A}=\{t\in \Delta\,/\, [\omega^{n-1}]_{A,t}^{n-2,n}=0\in H^{n-2,n}_A(X_t,\C)\}. $$
Thus, the derivative of the class $[\omega^{n-1}]_{A,t}^{n-2,n}\in H^{n-2,n}_A(X_t,\C)$ under the Gauss-Manin connection of the Hodge bundle $\Delta\ni t\longmapsto H^{2n-2}_{DR}(X_t,\C)$ in the direction of $t_i=0$ is the class $[v\lrcorner \omega^{n-1}]_A\in H^{n-2,n}_A(X,\C)$.\\

The proof of Proposition \ref{Prop} is complete.

\end{proof}

\begin{Obs}
For any Aeppli-Gauduchon class $[\omega^{n-1}]_A\in H^{n-1,n-1}_A(X,\C)$ and any class $[v]\in H^{0,1}(X,T^{1,0}X)$, the Aeppli class $[v\lrcorner \omega^{n-1}]_A\in H^{n-2,n}_A(X,\C)$ is independent of the choice of $d$-closed representative $\tilde{\omega}^{n-1}$ (only those which  differ from $\tilde{\omega}^{n-1}$ by a form that is both $d$-closed and Aeppli-exact) in the class $[\omega^{n-1}]_A\in H^{n-1,n-1}(X,\C)$.
\end{Obs}
\begin{proof}
Let $\tilde{\omega}^{n-1}:=\omega^{n-1}+\partial\phi+\bar{\partial}\psi$ be a $d$-closed representative of $[\omega^{n-1}]_A$. Using the surjectivity of the map $H^{n-1,n-1}_{BC}(X,\C) \longrightarrow H^{n-1,n-1}_{A}(X,\C)$, it is clear that for any Aeppli cohomology class $[w]_A\in H^{n-1,n-1}_A(X,\C)$, there exists a Bott-Chern cohomology class $[\tilde{w}]\in H^{n-1,n-1}_{BC}(X,\C)$ such that 
 $[w]_A=[\tilde{w}]_A$. This means that there exist $\alpha\in C^\infty_{n-2,n-1}(X,\C)$ and $\beta\in C^\infty_{n-1,n-2}(X,\C)$ such that
  $$w=\tilde{w}+\partial\alpha+\bar{\partial}\beta.$$
Putting $w:=\partial\phi+\bar{\partial}\psi$, implies that $\tilde{w}$ is a $d$-closed Aeppli-exact $(n-1,n-1)$-form and we have 
$$ [\tilde{\omega}^{n-1}-\tilde{w}]_A=[\tilde{\omega}^{n-1}]_A=[\omega^{n-1}]_A. $$ 
Let $\tilde{\omega}^{n-1}-\tilde{w}$ be another $d$-closed representative of $[\omega^{n-1}]_A$. Then 
$$[v\lrcorner(\tilde{\omega}^{n-1}-\tilde{w})]_A=[v\lrcorner\tilde{\omega}^{n-1}]_A-[v\lrcorner \tilde{w}]_A, \hspace{3ex} \forall [v]\in H^{0,1}(X,T^{1,0}X). $$
It is easy to check that $\tilde{w}$ is $\bar{\partial}$-exact, i.e. there exists $\varphi\in C^\infty_{n-1,n-2}(X,\C)$ such that $\tilde{w}=\bar{\partial}\varphi$. Hence
$$v\lrcorner \tilde{w}= v\lrcorner \bar{\partial}\varphi= \bar{\partial}(v\lrcorner \varphi)\in \mbox{ Im }\bar{\partial}\subset  \mbox{ Im }\partial+ \mbox{ Im }\bar{\partial}$$
where the last equation holds since $\bar{\partial}v=0$. Consequently, $[v\lrcorner \tilde{w}]_A=0$. Therefore
$$[v\lrcorner (\tilde{\omega}^{n-1}-\tilde{w})]_A =[v\lrcorner \tilde{\omega}^{n-1}]_A.$$
\end{proof}

Comparing the balanced co-polarisation to Gauduchon co-polarisation, we get the following result.
\begin{Prop}\label{Compr}
Let $(X,\omega)$ be a balanced Calabi-Yau $\partial\bar{\partial}$-manifold of complex dimension $n$. Then, the following identity holds:
\begin{equation}\label{compr}
H^{0,1}(X,T^{1,0}X)_{[\omega^{n-1}]}= H^{0,1}(X,T^{1,0}X)_{[\omega^{n-1}]_A}
\end{equation}

\noindent If we suppose that $(X,\omega)$ is a compact K\"ahler Calabi-Yau manifold, one obtains
\begin{equation}\label{compr 1}
 H^{0,1}(X,T^{1,0}X)_{[\omega]}= H^{0,1}(X,T^{1,0}X)_{[\omega^{n-1}]_A}.
\end{equation}
\end{Prop}
\begin{proof}
 Let $\omega$ be a balanced metric and let $[v]\in H^{0,1}(X,T^{1,0}X)_{[\omega^{n-1}]}$. So $[v\lrcorner \omega^{n-1}]=0\in H^{n-2,n}_{\bar{\partial}}(X,\C)$ (where $H^{n-2,n}_{\bar{\partial}}(X,\C)$ is Dolbeault cohomology group), i.e. $v\lrcorner \omega^{n-1}\in \mbox{ Im } \bar{\partial}\subset \mbox{ Im }\partial+\mbox{ Im }\bar{\partial} $. This proves that $[v]\in H^{0,1}(X,T^{1,0}X)_{[\omega^{n-1}]_A}$.\\
 
\noindent Conversely, suppose that $[v]\in H^{0,1}(X,T^{1,0}X)_{[\omega^{n-1}]_A}$. Then, $[v\lrcorner \omega^{n-1}]_A=0\in H^{n-2,n}_{A}(X,\C)$. In other words, there exist $(n-3,n)$-form $\alpha$ and $(n-2,n-1)$-form $\beta$ such that 
$$v\lrcorner\omega^{n-1} =\partial\alpha+\bar{\partial}\beta.$$
\noindent Meanwhile, $\partial\alpha$ is an $(n-2,n)$-form that is $d$-closed (for bidegree reasons) and $\partial$-exact. By the $\partial\bar{\partial}$-assumption, $\partial\alpha$ is $\bar{\partial}$-exact. Thus $v\lrcorner\omega^{n-1}$ is $\bar{\partial}$-exact.\\
\noindent On the other side, we have $\bar{\partial}(v\lrcorner\omega^{n-1})=0$ (for bidegree reasons).
 Therefore, $[v\lrcorner \omega^{n-1}]=0\in H^{n-2,n}_{\bar{\partial}}(X,\C)$ and the identity (\ref{compr}) follows.\\
 
In the case where $X$ is K\"ahler, due to a comparison in (\cite{Pop19}, Proposition 4.4), the balanced co-polarised deformations of $X$ coincide with polarised deformations
\begin{equation}\label{compr2}
 H^{0,1}(X,T^{1,0}X)_{[\omega]}= H^{0,1}(X,T^{1,0}X)_{[\omega^{n-1}]}.
\end{equation}
While, recall that if $X$ is K\"ahler, then $X$ is balanced. Therefore, the identity (\ref{compr}) holds on K\"ahler manifold. Accordingly, from the identities (\ref{compr2}) and (\ref{compr}), the identity (\ref{compr 1}) follows.
\end{proof}
\section{Primitive $(n-1,1)$-class and Weil-Petersson metrics on $\Delta_{[\omega^{n-1}]_A}$}\label{section prim+WP}
In this section, we discuss the notion of primitive classes of type $(n-1,1)$ on a compact Gauduchon Calabi-Yau $\partial\bar{\partial}$-manifold. We pursue by defining the Weil-Petersson metrics on $\Delta_{[\omega^{n-1}]_A}$ and ending with the comparison of the Weil-Petersson metric with the period map metric on $\Delta_{[\omega^{n-1}]_A}$ (cf.  Corollary \ref{comparison}).\\
Throughout this section, for any Aeppli-Gauduchon class $[\omega^{n-1}]_A\in H^{n-1,n-1}_A(X,\C)$, we consider that $\omega^{n-1}$  is its $\omega$-minimal $d$-closed representative.\\
\subsection{Primitive $(n-1,1)$-classes on Gauduchon manifolds}\label{sect prim}
The main purpose of this subsection is to study the existence of a primitive ($d$-closed) representative of a primitive $(n-1,1)$-class. For that, we shall start by defining the space of primitive classes.
\begin{Def}\label{def Gprim}
Let $(X,\omega)$ be a Gauduchon Calabi-Yau $\partial\bar{\partial}$-manifold with $\dim_\C X=n$. Consider that $u$ is the non-vanishing holomorphic $(n,0)$-form on $X$ and $[\omega^{n-1}]_A$ is an Aeppli-Gauduchon class on $X$. One defines {\bf the space of primitive classes of type} $(n-1,n)$ as the following:
$$H^{n-1,n}_{Gprim}(X,\C):= T_{[u]}\left(H^{0,1}(X,T^{1,0}X)_{[\omega^{n-1}]_A}\right)\subset H^{n-1,1}_A(X,\C)\simeq H^{n-1,1}(X,\C), $$
where $T_{[u]}:H^{0,1}(X,T^{1,0}X)\overset{\tilde{T}_{[u]}}{\longrightarrow}H^{n-1,1}(X,\C) \overset{\simeq}{\longrightarrow} H^{n-1,1}_A(X,\C) $, $[v]\longmapsto [v\lrcorner u]_{\bar{\partial}}\longmapsto[v\lrcorner u]_A$ \\
and $\tilde{T}_{[u]}$ is the Calabi-Yau isomorphism (cf. \cite{Pop19}, Lemma 3.3, p. 684) while $H^{0,1}(X,T^{1,0}X)_{[\omega^{n-1}]_A}$ is the vector subspace of $H^{0,1}(X,T^{1,0}X)$ defined in proposition \ref{Prop}. 
\end{Def}

\noindent Meaningly, for any $[v]\in H^{0,1}(X,T^{1,0}X)$, we have: 
$$ [v\lrcorner u]\in H^{n-1,1}_{Gprim}(X,\C) \hspace*{1.5ex} \Longleftrightarrow \hspace*{1.5ex} [v\lrcorner \omega^{n-1}]_A=0\in H^{n-2,n}_A(X,\C). $$
We denoted the space of primitive classes in the Gauduchon case by $H^{n-1,1}_{Gprim}(X,\C)$ to distinguish the notation with the space of primitive classes in the standard case and the balanced case. By proposition \ref{Compr}, we have
 $$H^{0,1}(X,T^{1,0}X)_{[\omega^{n-1}]_{\bar{\partial}}}=H^{0,1}(X,T^{1,0}X)_{[\omega^{n-1}]_A}$$
on any balanced Calabi-Yau $ \partial\bar{\partial}$-manifold. This means that:
$$ [v\lrcorner \omega^{n-1}]_{\bar{\partial}}=0 \in H^{n-2,n}_{\bar{\partial}}(X,\C) \Longleftrightarrow [v\lrcorner \omega^{n-1}]_A=0 \in H^{n-2,n}_A(X,\C)$$
for any $[v]\in H^{0,1}(X,T^{1,0}X)$. Therefore, the primitive $(n-1,1)$-classes defined by the Aeppli-Gauduchon class retain the properties of primitive $(n-1,1)$-classes defined by balanced class on any balanced Calabi-Yau $\partial\bar{\partial}$-manifold $(X,\omega)$. When the metric $\omega$ is K\"ahler, the definition of $H^{n-1,1}_{Gprim}(X,\C)$ coincides with both of the definition in the balanced case and with the standard definition of $H^{n-1,1}_{prim}(X,\C)$.\\

In the following, we will investigate the existence of a primitive $d$-closed representative of a primitive $(n-1,1)$-class.

\begin{Prop}
Let $(X,\omega)$ be a compact Hermitian manifold of dimension $n$. Suppose that $\alpha$ is a primitive $(n-1,1)$-form on $X$. The following three statements are equivalent.
\begin{enumerate}
\item[(i)] $\alpha$ is $d$-closed,
\item[(ii)] $\alpha$ is $\Delta_A$-harmonic,
\item[(iii)]$\alpha$ is $\Delta_{BC}$-harmonic.
\end{enumerate}
\end{Prop}
\begin{proof}
$(i)\Longleftrightarrow(ii)$ For any primitive pure-type $(n-1,1)$-form $\alpha$, we have:
$$
\begin{array}{ll}
d\alpha=0 &\Longleftrightarrow \partial\alpha=0 \hspace{1.5ex} \mbox{ and }  \hspace{1.5ex} \bar{\partial}\alpha=0 \\
&\overset{(a)}{\Longleftrightarrow} \bar{\partial}^*\alpha=0, \hspace{1.5ex} \partial^*\alpha=0  \hspace{1.5ex} \mbox{ and }  \hspace{1.5ex} \partial\bar{\partial}\alpha=0\\
&\overset{(b)}{\Longleftrightarrow} \alpha\in \ker \partial^*\cap \ker\bar{\partial}^*\cap \ker \partial\bar{\partial}=\ker\Delta_A
\end{array}
$$
where $(a)$ follows from (\cite{Pop19}, Lemma 4.12) and $(b)$ is (\ref{(1)}).\\
$(ii)\Longleftrightarrow(iii)$ Notice that by (\ref{(5)}), $\alpha$ is $\Delta_A$-harmonic is equivalent to $*\alpha$ is $\Delta_{BC}$-harmonic for any $(p,q)$-form $\alpha$. Moreover, since $\alpha$ is a primitive $(n-1,1)$-form. Then by (\ref{(4)}), $*\alpha$ is $\Delta_{BC}$-harmonic if and only if $\alpha$ is $\Delta_{BC}$-harmonic. As a consequence, being $\Delta_{BC}$-harmonic is equivalent to $\Delta_{A}$-harmonic.
\end{proof}
As a result, if an $(n-1,1)$-form is both primitive and $d$-closed, then it is $\Delta_{A}$-harmonic (and $\Delta_{BC}$-harmonic) form. Therefore, it remains to investigate whether the $\Delta_{A}$-harmonic representative of any primitive $(n-1,1)$-class is a primitive form on a Gauduchon Calabi-Yau manifold $(X,\omega)$.\\

Let $[v\lrcorner u]$ be a primitive $(n-1,1)$-class on $X$ with $[v]\in H^{0,1}(X,T^{1,0}X)$. By definition, we have:
$$ 
\begin{array}{ll}
[v\lrcorner u]\in H^{n-1,1}_{Gprim}(X,\C) &\Longleftrightarrow [v\lrcorner \omega^{n-1}]_A=0\in H^{n-2,n}_A(X,\C)\\
&\Longleftrightarrow v\lrcorner \omega^{n-1}\in \mbox{ Im } \partial+\mbox{ Im }\bar{\partial}.
\end{array}
$$

\noindent Now, let $v\lrcorner u$ be a $\Delta_A$-harmonic (i.e. $\Delta_A \left(v\lrcorner u\right)=0$) $(n-1,1)$-form. By (\cite{Pop19}, Lemma 4.10), $v\lrcorner u$ is primitive for $\omega$ if and only if $v\lrcorner \omega^{n-1}$ vanishes. Furthermore, one has $0=v\lrcorner \omega^{n-1}\in \mbox{ Im } \partial+\mbox{ Im }\bar{\partial}$, while $\ker\Delta_A$ is orthogonal to $\mbox{ Im } \partial+\mbox{ Im }\bar{\partial}$ (by the decomposition (\ref{(1)'})). Thereupon, $v\lrcorner \omega^{n-1}$ is $\Delta_A$-harmonic. Additionally, 
$$
\begin{array}{ll}
v\lrcorner \omega^{n-1}\in \ker\Delta_A &\overset{(a)}{\Longleftrightarrow} v\lrcorner \omega^{n-1}\in \ker \partial^*\cap \ker\bar{\partial}^*\cap \ker \partial\bar{\partial}\\
&\overset{(b)}{\Longleftrightarrow} v\lrcorner \omega^{n-1}\in \ker \partial^*\cap \ker\bar{\partial}^*\\
&\overset{(c)}{\Longleftrightarrow}  *\left( v\lrcorner \omega^{n-1}\right) \in \ker \partial\cap \ker\bar{\partial}\\
&\overset{(d)}{\Longleftrightarrow}  v\lrcorner \omega \in \ker \partial\cap \ker\bar{\partial}
\end{array}
$$
where $(a)$ holds by (\ref{(1)}), $(b)$ is trivial for bidegree reasons ($v\lrcorner \omega^{n-1}$ is of type $(n-2,n)$, then it is $\partial\bar{\partial}$-closed), $(c)$ is obtained by definition of $\partial^*$ and of $\bar{\partial}^*$ (\ref{(3)}) while $(d)$ follows since $v\lrcorner \omega^{n-1}=(n-1)!\, v\lrcorner \omega$ which holds from the fact that $v\lrcorner \omega$ is a primitive $(0,2)$-form and by the identity (\ref{prim form}). Hence,
\begin{equation}\label{obstract}
\Delta_A\left(v\lrcorner \omega^{n-1}\right)=0 \Longleftrightarrow d\left(v\lrcorner \omega \right)=0.
\end{equation}
The condition (\ref{obstract}) can hold if we consider $v\lrcorner \omega$ to be a $d$-closed representative of the Dolbeault (or Aeppli) cohomology class $[v\lrcorner \omega]_{\bar{\partial}}\in H^{0,2}_{\bar{\partial}}(X,\C)$ (or $[v\lrcorner \omega]_{A}\in H^{0,2}_A(X,\C)$). On the other side, the condition $v\lrcorner \omega$ being $\bar{\partial}$-closed $(0,2)$-form is equivalent to $\bar{\partial}\omega=0$, for any $[v]\in H^{0,1}(X,T^{1,0}X)$, which is equivalent to the standard case where $\omega$ is a K\"ahler form. Else, one can see no grounds for which $v\lrcorner \omega $ is $d$-closed even if we consider that $v\lrcorner u$ is $\Delta_A$-harmonic $(n-1,1)$-form.\\

Now, if we drop the $d$-closedness of the primitive representative (i.e. investigate only the existence of a primitive representative of a primitive $(n-1,1)$-class) and suppose that $[v\lrcorner u]\in H^{n-1,1}_{Gprim}(X,\C)$ is a primitive $(n-1,1)$-class. Then the primitivity of $v\lrcorner u $ is equivalent to the existence of a unique smooth vector field $\zeta$ of type $(1,0)$ and a unique primitive smooth $(1,2)$-form $v_0$ such that 
\begin{equation}\label{prim}
\left(v-\bar{\partial}\zeta\right)\lrcorner \omega^{n-1}=\bar{\partial}\left(\omega^{n-1}\wedge v_0\right)-\zeta \lrcorner \bar{\partial}\omega^{n-1}.
\end{equation}
Indeed, notice that we have:
$$ [v\lrcorner u]\in H^{n-1,1}_{Gprim}(X,\C) \hspace*{1.5ex} \Longleftrightarrow \hspace*{1.5ex} [v\lrcorner \omega^{n-1}]_A=0\in H^{n-2,n}_A(X,\C). $$
for any $[v]\in H^{0,1}(X,T^{1,0}X)$. While recall that the map $H^{p,q}_{\bar{\partial}}(X,\C)\longrightarrow H^{p,q}_A(X,\C)$ is an isomorphism on a $\partial\bar{\partial}$-manifold. Then, the vanishing of the Aeppli class $[v\lrcorner \omega^{n-1}]_A$ implies the vanishing of the Dolbeault class $[v\lrcorner \omega^{n-1}]_{\bar{\partial}}$. Thus, $v\lrcorner \omega^{n-1}$ is $\bar{\partial}$-exact, i.e. there exists an $(n-2,n-1)$-form $\alpha$ such that $v\lrcorner \omega^{n-1}=\bar{\partial}\alpha$. As in \cite{Pop19}, $\alpha$ is of the shape:
$$\alpha=\omega^{n-3}\wedge v_0+\zeta\lrcorner \omega^{n-1}, $$
where $v_0$ is a unique primitive $C^\infty$ $(1,2)$-form and $\zeta$ is a unique $C^\infty$ vector field of type $(1,0)$. Applying $\bar{\partial}$ to $\alpha$, one obtains (\ref{prim}). To summerize, since $[v-\bar{\partial}\zeta]\in H^{0,1}(X,T^{1,0}X)$, we have:
$$
\begin{array}{ll}
v\lrcorner u \mbox{ is primitive for } \omega &\Longleftrightarrow v\lrcorner \omega^{n-1}=0\\
&\Longleftrightarrow \bar{\partial}\left(\omega^{n-1}\wedge v_0\right)-\zeta \lrcorner \bar{\partial}\omega^{n-1}=0.
\end{array}
$$
Neverthless, one has no grounds for which the last condition holds even if we exploit the fact that $\bar{\partial}\omega^{n-1}$ vanishes since we have chosen $\omega^{n-1}$  to be the $\omega$-minimal $d$-closed representative of the Aeppli-Gauduchon class $[\omega^{n-1}]_A\in H^{n-1,n-1}_A(X,\C)$. Therefore, on a Gauduchon non-K\"ahler Calabi-Yau $\partial\bar{\partial}$-manifold, a primitive $(n-1,1)$-class may not be represented by a primitive form.
\subsection{Weil-Petersson metrics on $\Delta_{[\omega^{n-1}]_A}$}\label{WP metrics}
Let $X$ be an arbitrary Gauduchon Calabi-Yau $\partial\bar{\partial} $-manifold of dimension $n$. As discussed in the Introduction and in Section (\ref{copolar}), all the fibres $(X_t)_{t\in \Delta}$ are Gauduchon Calabi-Yau $\partial\bar{\partial}$-manifolds as well for $t\in\Delta$ close enough to $0$. The subspace $H^{0,1}(X,T^{1,0}X)_{[\omega^{n-1}]_A}$ is the tangent space of the base space $\Delta_{[\omega^{n-1}]_A}$ of the local universal family $(X_t)_{t\in\Delta_{[\omega^{n-1}]_A}}$ of deformations of $X$ that are co-polarised by the Aeppli-Gauduchon class $[\omega^{n-1}]_A\in H^{n-1,n-1}_A(X,\C)$, namely
$$T_t\Delta_{[\omega^{n-1}]_A}\simeq H^{0,1}(X_t,T^{1,0}X_t)_{[\omega^{n-1}]_A} \simeq H^{n-1,1}_{Gprim}(X_t,\C), \hspace*{3ex} t\in\Delta_{[\omega^{n-1}]_A}.$$
One can define the Weil-Petersson metrics on $\Delta_{[\omega^{n-1}]_A}$ as follows
\begin{Def}\label{def of WP metrics}
Let $(u_t)_{t\in\Delta}$ be a fixed holomorphic family of non-vanishing holomorphic $n$-forms on the fibres $(X_t)_{t\in\Delta}$ and let $(\omega_t)_{t\in\Delta_{[\omega^{n-1}]_A}}$ be a smooth family of Gauduchon metrics on the fibres  $(X_t)_{t\in\Delta_{_{[\omega^{n-1}]_A}}}$ such that $\omega_t^{n-1}\in\{\omega^{n-1}\}$ for any $t$ and $\omega_0=\omega$. {\bf The Weil-Petersson metrics} $g^{(1)}_{WP}$ and $g^{(2)}_{WP}$ are defined on $\Delta_{[\omega^{n-1}]_A}$ by
$$ g^{(1)}_{WP}([v_t],[w_t])\,:=\,\dfrac{\ll v_t,w_t \gg}{\int_{X_t}dV_{\omega_t^{n-1}}}$$
$$ g^{(2)}_{WP}([v_t],[w_t])\,:=\,\dfrac{\ll v_t\lrcorner u_t,w_t\lrcorner u_t \gg}{i^{n^2}\int_{X_t} u_t\wedge \overline{u}_t}$$
for any $t \in \Delta_{_{[\omega^{n-1}]_A}}$, any $v_t\in [v_t]\in  H^{0,1}(X_t,T^{1,0}X_t)_{[\omega^{n-1}]_A}$ such that $v_t\lrcorner u_t$ is the $\omega_t$-minimal $d$-closed representative of the class $[v_t\lrcorner u_t]_A\in  H^{n-1,1}_{A}(X_t,\C)$ and any $w_t\in [w_t]\in  H^{0,1}(X_t,T^{1,0}X_t)_{[\omega^{n-1}]_A}$ such that $w_t\lrcorner u_t$ is the $\omega_t$-minimal $d$-closed representative of the class $[w_t\lrcorner u_t]_A\in  H^{n-1,1}_{A}(X_t,\C)$, while $dV_{\omega_t}:= \dfrac{\omega^{n}}{n!}$ and $\ll.,.\gg$ stands for the $L^2$ scalar product induced by $\omega_t$.
\end{Def}
\noindent The smooth positive definite $(1,1)$-forms on $ \Delta_{_{[\omega^{n-1}]_A}}$ associated with $g^{(1)}_{WP}$ and $g^{(2)}_{WP}$ are denoted by 
$$\tilde{\omega}^{(1)}_{WP}>0 \hspace*{2ex} \mbox{ and } \hspace*{2ex} \tilde{\omega}^{(2)}_{WP}>0  \hspace*{2ex}\mbox { on }\Delta_{_{[\omega^{n-1}]_A}}. $$
Similarly to the balanced case (cf, \cite{Pop19}, Lemma 3.2), if the Gauduchon metrics can be chosen such that $Ric(\omega_t)=0$ for all $t\in \Delta_{_{[\omega^{n-1}]_A}}$, then 
$$\ll v_t\lrcorner u_t,w_t\lrcorner u_t \gg \, = \, \ll v_t,w_t \gg. $$ 
Thus $$\tilde{\omega}^{(1)}_{WP}\,=\,\tilde{\omega}^{(2)}_{WP}\hspace*{2ex}\mbox { on }\Delta_{_{[\omega^{n-1}]_A}}.$$
Now, let $\gamma$ be the period map metric on $\Delta$ defined  in \cite{Pop19}. The K\"ahler metric  $\gamma$  is independent of the choice of any metrics on $(X_t)_{t\in\Delta}$ and given by:
$$
\gamma_t([v_t],[v_t])=\left\{
\begin{array}{rl}
\dfrac{-\int_X (v_t\lrcorner u_t)\wedge \overline{(v_t\lrcorner u_t)}}{i^{n^2}\int_X u_t\wedge \overline{u}_t}, & \mbox{ if } n \mbox{ is even }\\
\\
\dfrac{-i\int_X (v_t\lrcorner u_t)\wedge \overline{(v_t\lrcorner u_t)}}{i^{n^2}\int_X u_t\wedge \overline{u}_t}, & \mbox{ if } n \mbox{ is odd }
\end{array}
\right.
$$
for every $[v_t]\in H^{0,1}(X,T^{1,0}X) $ and every $t\in \Delta$. Furthermore, $\gamma_t([v_t],[v_t])$ is independent of the choice of representative $v_t$ in the class $[v_t]\in H^{0,1}(X,T^{1,0}X)$ such that $v_t\lrcorner u_t$ is $d$-closed and independent of the choice of holomorphic family $(u_t)_{t\in\Delta}$ of $J_t$-holomorphic $n$-forms. \\

To compare the Weil-Petersson metric $\tilde{\omega}^{(2)}_{WP}$ with the period map metric $\gamma$ on $\Delta_{[\omega^{n-1}]_A}$, one needs the following
\begin{Prop}
Suppose that $X$ is a compact Gauduchon Calabi-Yau $\partial\bar{\partial}$-manifold with $\dim_\C X=n$. For any $t\in\Delta_{[\omega^{n-1}]_A}$ on the open subset of local deformations of $X$ co-polarised by a given Aeppli-Gauduchon class $[\omega^{n-1}]_A\in H^{n-1,n-1}_A(X,\C)\subset H^{2n-2}(X,\C)$, the metrics $\gamma$ and $g^{(2)}_{WP}$ are stated as follows:
$$
\begin{array}{ll}
& g^{(2)}_{WP}([v_t],[v_t])= \dfrac{\parallel v'_t\lrcorner u_t \parallel^2+2\parallel \zeta_t\parallel^2}{i^{n^2}\int_X u_t\wedge \overline{u}_t}\\
&\\
& \gamma_t([v_t],[v_t])= \dfrac{\parallel v'_t\lrcorner u_t \parallel^2-2\parallel \zeta_t\parallel^2}{i^{n^2}\int_X u_t\wedge \overline{u}_t}
\end{array}
$$
with $v_t\in [v_t]\in H^{0,1}(X_t,T^{1,0}X_t)_{[\omega^{n-1}]_A}$ such that $v_t\lrcorner u_t$ is the $\omega_t$-minimal $d$-closed representative of the class $[v_t\lrcorner u_t]\in H^{n-1,1}_A(X_t,\C)$ and $\omega_t\in\{\omega^{n-1}\}$ are Gauduchon metrics in the co-polarising Aeppli-Gauduchon class given previously, while $v'_t\lrcorner u_t$ is a primitive $(n-1,1)$-form and $\zeta_t$ is an $(n-2,0)$-form such that $v_t\lrcorner u_t=v'_t\lrcorner u_t+\omega\wedge\zeta_t$ given by the Lefschetz decomposition in [\cite{Voi02}, Proposition 6.22, page 147].
\end{Prop}
\noindent The above proposition is a similar finding to (\cite{Pop19}, Theorem 5.10) and both can be proved in a similar way.\\

As a consequence, the Hermitian metric $\tilde{\omega}^{(2)}_{WP}$ on $\Delta_{[\omega^{n-1}]_A}$ defined by $g^{(2)}_{WP}$ is bounded below by the K\"ahler metric $\gamma$ as shown in the following
\begin{Cor}\label{comparison}
For every $[v_t]\in H^{0,1}(X_t,T^{1,0}X_t)_{[\omega^{n-1}]_A}\setminus\{0\}$, 
$$\left(g^{(2)}_{WP}-\gamma \right)\left( [v_t],[v_t]\right)=\dfrac{4 \parallel \zeta_t\parallel^2}{i^{n^2}\int_X u_t\wedge \overline{u}_t}\geq 0, \hspace*{3ex} t\in\Delta_{[\omega^{n-1}]_A}. $$
\end{Cor}
Eventually, on a Gauduchon (non-K\"ahler) Calabi-Yau $\partial\bar{\partial}$-manifold, the metric $\tilde{\omega}^{(2)}_{WP}$ can coincide with the K\"ahler metric $\gamma$ on $\Delta_{[\omega^{n-1}]_A}$ if the primitive $(n-1,1)$-class $[v_t\lrcorner u_t]\in H^{n-1,1}_{Gprim}(X,\C)$ can be represented by a form that is both primitive and $d$-closed.

\section{Deformations of $p$-SKT $h$-$\partial\bar{\partial}$-manifolds}
We will start this section by recalling some notions that we will need in the sequel.

\begin{Def}\label{recall defn}
Let $X$ be a compact complex manifold with $\dim_\C X= n$. For every $p=0,\cdots,n$, let $\Omega$ be a smooth strictly weakly positive (see e.g. [\cite{Dem}, Chapter III] or [\cite{B19}, Definition 2.1])  $(p,p)$-form and $\omega$ a positive-definite $C^\infty$ $(1,1)$-form on $X$.
\begin{enumerate}
\item  $\Omega$ is said to be {\it $p$-SKT (or $p$-pluriclosed)} if $\partial\bar{\partial}\Omega=0$.
\item $\Omega$ is called {\it $p$-Hermitian symplectic ($p$-HS) form} if there exist $\Omega^{i,2p-i}\in C^\infty_{i,2p-i}(X,\C)$, $i\in\{0,\cdots,p-1\}$, such that
$$d\left(\sum_{i=0}^{p-1}\Omega^{i,2p-i}+\Omega+\sum_{i=0}^{p-1}\overline{\Omega^{i,2p-i}}\right)=0. $$
\item  $\omega$ is called {\it strongly Gauduchon (sG)} (cf. \cite{Pop09}) if $\partial\omega^{n-1}$ is $\bar{\partial}$-exact.
\item  X is called {\it $p$-SKT, $p$-HS and  sG manifold} if there exists a {\it $p$-SKT form, $p$-HS form and sG metric} respectively on $X$.
\item Let $h\in\R\setminus{\{0\}}$  be an arbitrary constant. A compact complex manifold $X$ with $dim_\C X = n$ is said to be an $h$-$\partial\bar{\partial}$-manifold (\cite{BP18}, Definition 1.5) if $X$ satisfies the $h$-$\partial\bar{\partial}$-lemma, i.e. for
every $k =0,1,\cdots,2n$ and every $k$-form $u\in\ker d_h \cap \ ker d_{-h^{-1}}$ , the following exactness conditions are equivalent:
$$u \in \mbox { Im }d_h \Longleftrightarrow u \in \mbox { Im }d_{-h^{-1}} \Longleftrightarrow u \in \mbox { Im } d \Longleftrightarrow u \in \mbox { Im } (d_h d_{-h^{-1}}) = Im (\partial\bar{\partial}).$$
\end{enumerate}
\end{Def}
 
Now, one defines some other new notions that we call {\it $hp$-Hermitian symplectic ($hp$-HS) form} and {\it $h$-strongly Gauduchon ($h$-sG) metric } on a compact complex manifold of complex dimension $n$.

\begin{Def}\label{hpHS defn}
Let $X$ be a compact complex manifold of dimension $n$ and $\omega$ a positive-definite $C^\infty$ $(1,1)$-form on $X$. For any $p=0,\cdots,n$, let $\Omega$ be a smooth strictly weakly positive  $(p,p)$-form on $X$. For every $h\in\mathbb{R}\setminus{\{0\}}$,

\begin{enumerate}
\item $\omega$ is called {\bf $h$-strongly Gauduchon ($h$-sG) metric } if there exists $\Omega^{n-2,n}\in C^\infty_{n-2,n}(X,\C)$ such that 
$$ d_h\bigg(\dfrac{1}{h}\Omega^{n-2,n}+\omega^{n-1}+h\overline{\Omega^{n-2,n}}\bigg)=0 .$$
\item $\Omega$ is called {\bf $hp$-Hermitian symplectic ($hp$-HS) form} if there exists $\Omega^{i,2p-i}\in C^\infty_{i,2p-i}(X,\C)$ with $i=0,\cdots,p-1$ such that 
$$ d_h\bigg(\sum_{i=0}^{p-1}\Omega^{i,2p-i}+\Omega+\sum_{i=0}^{p-1}\overline{\Omega^{i,2p-i}}\bigg)=0 .$$
\item $X$ is said to be {\bf $h$-sG (resp. $hp$-HS) manifold}  if there exists an {\bf $h$-sG metric (resp. $hp$-HS form)} on $X$.
\end{enumerate}
\end{Def}

\begin{Obs}Let $X$ be a compact complex manifold with $\dim_\C X=n$. For every $h\in \R\setminus{\{0\}}$, suppose that $\Omega$ is an $hp$-HS form on $X$. In the case where $p=n-1$, $\omega$ is either balanced or sG metric on $X$ where $\omega$ is a positive-definite $C^\infty$ $(1,1)$-form on $X$. 
\end{Obs}
\begin{proof}
Let $\Omega$ be a smooth strictly weakly positive $(n-1,n-1)$-form. Note that the notion of strict weak positivity and the usual positivity definiteness for $(1,1)$-form and $(n-1,n-1)$-form are equivalent \cite{Dem}, so there exists a unique smooth $(1, 1)$-form 
 $\omega> 0$ (c.f. \cite{Mic82}) on $X$ such that $\Omega=\omega^{n-1}$.
 
\noindent Suppose that $\Omega$ is an $hp$-HS form on $X$ and $p=n-1$. So, there exists $\Omega^{n-2,n}\in C^\infty_{n-2,n}(X,\C)$ such that 
$$d_h(\Omega^{n-2,n}+\omega^{n-1}+\overline{\Omega^{n-2,n}})=0 \hspace{2ex} \Longrightarrow \hspace{2ex} \partial\omega^{n-1}=-\dfrac{1}{h} \bar{\partial} \overline{\Omega^{n-2,n}} \, \mbox{ and } \, \bar{\partial}\omega^{n-1}=-h\partial\Omega^{n-2,n}, \hspace{2ex} \forall h\in\R\setminus{\{0\}}. $$
By conjugation, it is clear that $\bar{\partial}\omega^{n-1}= -h\partial\Omega^{n-2,n}=-\dfrac{1}{h}\partial\Omega^{n-2,n}$. This implies that $ (h-\dfrac{1}{h})\partial\Omega^{n-2,n}=0.$\\
If $h\neq \dfrac{1}{h}$, then $\partial\omega^{n-1}=0$ and $\bar{\partial}\omega^{n-1}=0$. Therefore $d\omega^{n-1}=0$, i.e. $\omega$ is a {\it balanced metric} on $X$. \\
If $\partial\Omega^{n-2,n}\neq 0$, we must have $h=\{1,-1\}$. when $h=1$, $d_1=d$. So $d(\Omega^{n-2,n}+\omega^{n-1}+\overline{\Omega^{n-2,n}})=0$. Thus $\omega$ is an sG metric on $X$. While when $h=-1$, we have
$$
\begin{array}{ll}
d_{-1}(\Omega^{n-2,n}+\omega^{n-1}+\overline{\Omega^{n-2,n}})=0 & \Longrightarrow -\partial\Omega^{n-2,n}-\partial\omega^{n-1}+\bar{\partial}\omega^{n-1}+\bar{\partial}\overline{\Omega^{n-2,n}}=0\\
&\Longrightarrow \partial\omega^{n-1}\in \mbox{ Im }\bar{\partial}.
\end{array}
$$
It follows that if $h=\{1,-1\}$, then $\omega$ is an sG metric on $X$.\\
The proof is complete.
\end{proof}

 In the following result, we indicate that the notions of $h$-sG and sG are equivalent on any compact complex manifold of complex dimension $n$.

\begin{Prop}\label{relations}
Let $X$ be a compact complex manifold with $\dim_{\C} X=n$. Let $h\in \R\setminus{\{0\}}$ and suppose that $\omega$ is a Hermitian metric on $X$.
\begin{enumerate}
\item $X$ is an $h$-sG manifold if and only if $X$ is an sG manifold.
\item If $\omega$ is an $h$-sG metric on $X$, then
$$ (i)\hspace{1ex} d_h\omega^{n-1}\in \mbox{ Im }d_{-h^{-1}}; \hspace{1cm} (ii)\hspace{1ex} d_{-h^{-1}}\omega^{n-1}\in \mbox{ Im }d_h.$$
\item Suppose that $ d_{-h^{-1}}\omega^{n-1}$ is $d_h$-exact and that the following special case of the $h$-$\partial\bar{\partial}$-lemma:
\begin{equation}\label{hddbarlem}
\forall u\in \ker\,d_h \cap \ker\,  d_{-h^{-1}}; \hspace*{1cm} u\in \mbox{ Im } d_{-h^{-1}} \Longrightarrow u\in \mbox{ Im } d_h d_{-h^{-1}}
\end{equation}
holds. Then $\omega$ is an $h$-sG metric on $X$.
\end{enumerate}
\end{Prop}

\begin{proof}
\begin{enumerate}
\item Suppose that $X$ is an $h$-sG manifold. There exists an $h$-sG metric $\omega$ on $X$, i.e. $\exists \Omega^{n-2,n}\in C^\infty_{n-2,n}(X,\C)$ such that 
$$d_h(\dfrac{1}{h}\Omega^{n-2,n}+\omega^{n-1}+h\overline{\Omega^{n-2,n}})=0, \hspace{2ex} \mbox{ for any } h\in\R\setminus{\{0\}}.$$
Replacing $d_h=h\partial+\bar{\partial}$ in the above equation and computing, one obtains
$$ \partial\Omega^{n-2,n}+h\partial\omega^{n-1}+\bar{\partial}\omega^{n-1}+h\bar{\partial}\overline{\Omega^{n-2,n}} =0$$
This means that $\partial\omega^{n-1}=-\bar{\partial}\overline{\Omega^{n-2,n}}$ for any real non-zero constant $h$. Therefore, $\omega$ is an sG metric on $X$, so $X$ is an sG manifold.\\

\noindent Conversely, consider that there exists a smooth (n-2,n)-form $\Omega^{n-2,n}$ such that $\Omega=\Omega^{n-2,n}+\omega^{n-1}+\overline{\Omega^{n-2,n}}$ is a $d$-closed $(2n-2)$-form on $X$. By the isomorphism (\ref{DRdh}), we have $d_h(\theta_h \Omega)=0$ for any real non-zero constant $h$, i.e.
$$d_h(h^{n-2}\Omega^{n-2,n}+h^{n-1}\omega^{n-1}+h^{n}\overline{\Omega^{n-2,n}})=0 $$
It follows that $\dfrac{1}{h}\Omega^{n-2,n}+\omega^{n-1}+h\overline{\Omega^{n-2,n}}$ is $d_h$-closed (since $h\neq 0$). Thus $X$ is an $h$-sG manifold.
\item  $(ii)$ follows from (i) by replacing $h$ with $-\dfrac{1}{h}$. It suffices to prove $(i)$ for an arbitrary $h\neq 0$.\\
Let $\omega$ be an $h$-sG metric on $X$. Then
$$
\begin{array}{ll}
d_h\omega^{n-1}&=h\partial\omega^{n-1}+\bar{\partial} \omega^{n-1}=-h\bar{\partial}\overline{\Omega^{n-2,n}}-\partial\Omega^{n-2,n}= -\bar{\partial}(h\overline{\Omega^{n-2,n}})-\dfrac{1}{h}\partial(h\Omega^{n-2,n})\\
&=d_{-\frac{1}{h}}(h\Omega^{n-2,n}-h\overline{\Omega^{n-2,n}}).
\end{array}
$$
\noindent We conclude that $d_h\omega^{n-1}$ is $d_{-\frac{1}{h}}$-exact for any real non-zero constant $h$ if $\omega$ is an $h$-sG metric on $X$.
\item  By assumption, $d_{-h^{-1}}\omega^{n-1}$ is $d_h$-exact. So $d_{-h^{-1}}\omega^{n-1}$ is $d_h$-closed, $d_{-h^{-1}}$-closed and $d_{-h^{-1}}$-exact. It follows that $d_{-h^{-1}}\omega^{n-1}$ is $d_hd_{-h^{-1}}$-exact or equivalently $\partial\bar{\partial}$-exact, i.e. there exists an $(2n-3)$-form $\alpha$ such that 
\begin{eqnarray*}
d_{\frac{-1}{h}}\omega^{n-1}&=&-\partial\bar{\partial}\alpha\overset{(1)}{=}-\partial\bar{\partial}(\alpha^{n-3,n}+\alpha^{n-2,n-1}+\alpha^{n-1,n-2}+\alpha^{n,n-3})\\
&\overset{(2)}{=}&-\partial\bar{\partial}(\alpha^{n-2,n-1}+\alpha^{n-1,n-2})\\
&\overset{(3)}{=}&-\frac{1}{h}\partial\omega^{n-1}+\bar{\partial}\omega^{n-1}
\end{eqnarray*}
where $(1)$ is the splitting of the $(2n-3)$-form $\alpha$ into pure-type, $(2)$ is obtained since we have $\bar{\partial}\alpha^{n-3,n}=0$ and $\partial\alpha^{n,n-3}=0$ (for bidegree reasons), while $(3)$ is the definition of $d_{-h^{-1}}\omega^{n-1}$. One obtains:
$$
\left\{
\begin{array}{rl}
-\dfrac{1}{h}\partial\omega^{n-1}=-\partial\bar{\partial}\alpha^{n-1,n-2}\\
\bar{\partial}\omega^{n-1}=-\partial\bar{\partial}\alpha^{n-2,n-1}
\end{array}
\right.
$$
Hence the latter equation implies, by conjugation, that $\partial\omega^{n-1}=-\bar{\partial}\partial\overline{\alpha^{n-2,n-1}}$. It follows that
\begin{eqnarray*}
d_h\omega^{n-1}&=& h\partial\omega^{n-1}+\bar{\partial}\omega^{n-1}=-h\bar{\partial}(\partial\overline{\alpha^{n-2,n-1}})-\partial(\bar{\partial}\alpha^{n-2,n-1})\\
&=&-d_h(h\partial\overline{\alpha^{n-2,n-1}})-d_h(\frac{1}{h}\bar{\partial}\alpha^{n-2,n-1})\\
&=&-d_h(h\partial\overline{\alpha^{n-2,n-1}}+\frac{1}{h}\bar{\partial}\alpha^{n-2,n-1}).
\end{eqnarray*}
This proves that the $(2n-2)$-form $h\partial\overline{\alpha^{n-2,n-1}}+\omega^{n-1}+\frac{1}{h}\bar{\partial}\alpha^{n-2,n-1}$ is $d_h$-closed with $\partial\overline{\alpha^{n-2,n-1}}$ and $\bar{\partial}\alpha^{n-2,n-1}$ are forms of bidegree $(n,n-2)$ and  $(n-2,n)$ respectively. Therefore, $\omega$ is an $h$-sG metric on $X$.

\end{enumerate}
The proof is complete.
\end{proof}
As a consequence, one has the following
\begin{Prop}
Let $X$ be a compact complex manifold with $\mbox{dim}_\C X=n$. Let $\omega$ be an arbitrary Hermitian metric on $X$ and suppose that $X$ satisfies the special case of the $h$-$\partial\bar{\partial}$-lemma (\ref{hddbarlem}), then the following statements are equivalent:
\begin{enumerate}
\item $\omega$ is an $h$-sG metric;
\item $d_{\frac{-1}{h}}\omega^{n-1}\in \mbox{ Im }d_h$;
\item $ \omega$ is an $h$-Gauduchon metric (i.e.  $d_hd_{\frac{-1}{h}}\omega^{n-1}=0)$;
\item $ \omega$ is a Gauduchon metric (i.e.  $\partial\bar{\partial}\omega^{n-1}=0)$.
\end{enumerate}
\end{Prop}

\begin{proof}
$(1)\Longrightarrow (2)$ is the assertion $(2)$ of the proposition \ref{relations}.\\
$(2)\Longrightarrow (1)$ is the assertion $(3)$ of the proposition \ref{relations}.\\
$(2)\Longrightarrow (3)$ is trivial.\\
$(3)\Longrightarrow (2)$ Suppose that $d_hd_{-h^{-1}}\omega^{n-1}=0$, so $d_{-h^{-1}}\omega^{n-1}$ is $d_h$-closed, $d_{-h^{-1}}$-closed and $d_{-h^{-1}}$-exact. Hence, by the special case the $h$-$\partial\bar{\partial}$-lemma (\ref{hddbarlem}),  $d_{-h^{-1}}\omega^{n-1}$ is $d_h$-exact.\\
$(3)\Longleftrightarrow (4)$ is trivial since $d_hd_{-h^{-1}}=(h+\frac{1}{h})\partial\bar{\partial}$, i.e. $\ker d_hd_{-h^{-1}}=\ker\partial\bar{\partial}$.
\end{proof}

 We can now observe the deformation-openness of the $h$-sG property.
\begin{Cor}
Let $\pi: \mathcal{X}\longrightarrow \Delta$ be a holomorphic family of compact complex manifolds of dimension $n$. Fix an arbitrary constant $h\in\R\setminus{\{0\}}$.\\
If $X_0$ is an $h$-sG manifold, then $X_t$ is an $h$-sG manifold for all $t\in\Delta$ sufficiently close to 0.
\end{Cor}

\begin{proof}
The proof is trivial since the $h$-sG and sG properties are equivalent and the notion of sG being deformation open (see e.g. \cite{Pop10} or \cite{B19} for p=n-1 ). Otherwise,
$$
\begin{array}{ll}
X_0 \mbox{ is an } h\mbox{-sG manifold} &\Longleftrightarrow  X_0 \mbox{ is an sG manifold}\\
&\Longrightarrow X_t \mbox{ is an sG manifold}, \hspace{2ex} \forall t\in\Delta, t\sim 0\\
&\Longleftrightarrow X_t \mbox{ is an } h\mbox{-sG manifold}, \hspace{2ex} \forall t\in\Delta, t\sim 0.
\end{array}
$$
\end{proof}

As discussed above, the map from the De Rham cohomology to the $d_h$-cohomology is defined as in (\ref{DRdh}). Here, we will give another definition of the $d_h\longrightarrow $ De Rham-map  on any compact complex $h$-$\partial\bar{\partial}$-manifold.

\begin{Prop}\label{dhDR}
Let $h\in\R\setminus{\{0\}}$ be an arbitrary constant. Let $X$ be a compact complex $h$-$\partial\bar{\partial}$-manifold with $\dim_\C X=n$.
\begin{enumerate}
\item Every $d_h$-cohomology class contains a $d$-closed representative.
\item Let $k\in\{0,\cdots,2n\}$. The following map 
$$
\begin{array}{ll}
F: &H^k_{d_h}(X,\C) \longrightarrow H^k_{DR}(X,\C)\\
& \hspace*{1cm} [\alpha]_h\longmapsto \{\alpha\}
\end{array}
$$
is well defined. Moreover, $F$ is an isomorphism.
\end{enumerate}
\end{Prop}

\begin{proof}
1. We want to prove the existence of $\beta\in C^\infty_{k-1}(X,\C)$ such that 
$$ d(\alpha+d_h\beta)=0 \hspace*{2ex} \mbox{ with } \hspace{2ex} d_h\alpha=0. $$
This is equivalent to find $\beta$ such that $d\alpha=-dd_h\beta$ with $d_h\alpha=0$.\\
 Notice that $d_{-h^{-1}}d\alpha=\dfrac{h+1}{h(h-1)}dd_h\alpha=0$ for every $h\in\R\setminus{\{0,1\}}$ since $d_h\alpha=0$. $d\alpha$ is a $d_h$-closed, $d_{-h^{-1}}$-closed and $d$-exact $k$-form for every $h\in\R\setminus{\{0,1\}}$. Using the $h$-$\partial\bar{\partial}$ assumption, $d\alpha$ is $d_hd_{-h^{-1}}$-exact. Therefore $d\alpha$ is $dd_h$-exact. This proves the existence of a $(k-1)$-form $\beta$ such that $d(\alpha+d_h\beta)=0$ for every $h\in\R\setminus{\{0,1\}}$. The case where $h=1$ is trivial since $d_1=d$. One can conclude that 
$$ d(\alpha+d_h\beta)=0, \hspace*{1cm} \forall h\in\R\setminus{\{0\}}. $$
2. The map $F$ is independant of choice of $d$-closed representative of the $d_h$-cohomology class $[\alpha]_h$. Indeed:\\
Let $\alpha$, $\beta \in C^\infty_k(X,\C)$ be  $d$-closed forms representing the same $d_h$-cohomology class, i.e. 
$$d\alpha=d\beta=0 \hspace*{2ex} \mbox{ and } \hspace*{2ex} \alpha-\beta=d_h u$$
for some $(k-1)$-form  $u$. Then $0=d(\alpha-\beta)=dd_h u= \dfrac{h(h-1)}{h^2+1} d_{-h^{-1}}d_h u$. This means that $d_h u\in\ker d_h \cap \ker d_{-h^{-1}} \cap \mbox{ Im } d_h $. The $h$-$\partial\bar{\partial}$ assumption implies that $d_h u$ is $d$-exact. In other words, there exists a $(k-1)$-form $\gamma$ such that  $\alpha-\beta=d\gamma.$ Thus $F$ is well defined.\\

To show the injectivity of $F$, suppose that $\alpha$ is a $d_h$-closed $k$-form such that $\alpha=d u$ where $u$ is a $(k-1)$-form. Since $d_{-h^{-1}}\alpha=d_{-h^{-1}}d u=\dfrac{h+1}{h(1-h)} d_h du=\dfrac{h+1}{h(1-h)} d_h \alpha=0$ for any $h\in\R\setminus{\{0,1\}}$ whereof for $h=1$ we have $d_1=d$, $\alpha$ is $d_h$-closed, $d_{-h^{-1}}$-closed and $d$-exact for any $h\in\R\setminus{\{0\}}$. Due to the $h$-$\partial\bar{\partial}$-lemma, $\alpha$ is $d_h$-exact. Hence $F$ is injective.\\

It remains to prove that $F$ is surjective. Let $\alpha$ be a $d$-closed $k$-form on $X$. We need to prove the existence of a $(k-1)$-form $v$ such that $d_h(\alpha+dv)=0$. We will then have $\{\alpha\}=\{\alpha+dv\}=F([\alpha+dv]_h)$. Note that $d\alpha=0$ implies that $d_{-h^{-1}}d_h\alpha=-(h+\dfrac{1}{h})\partial d\alpha=0.$
 Since $d_h\alpha\in \mbox{ Im }d_h$ is $d_h$-closed and $d_{-h^{-1}}$-closed, then it must be $d_hd_{-h^{-1}}$-exact (due to the $h$-$\partial\bar{\partial}$-lemma) which is equivalent to the $d_hd$-exactness.
\end{proof}

For any $h\in\R\setminus{\{0\}}$, the $h$-twisted analogue of the Aeppli cohomology (\ref{Aeppli cohom})  (called $h$-Aeppli cohomology) is defined, in \cite{BP18}, in a given total degree $k\in\{0,\cdots,2n\}$ as the following
$$ H^k_{h,A}(X,\C)= \dfrac{\ker\{ d_hd_{-h^{-1}}: C^\infty_{k}(X,\C)\longrightarrow C^\infty_{k+2}(X,\C)\}}{\mbox{Im }\{d_h: C^\infty_{k-1}(X,\C)\longrightarrow C^\infty_{k}(X,\C)\}+\mbox{Im }\{d_{-h^{-1}}: C^\infty_{k-1}(X,\C)\longrightarrow C^\infty_{k}(X,\C)\}}, $$

\noindent and for all $p,q=0,\cdots,n$ with $p+q=k$, the following identity 
\begin{equation}\label{h-Aeppli decomp}
H^k_{h,A}(X,\C)=\bigoplus_{p+q=k} H^{p,q}_A(X,\C)
\end{equation}
holds on any compact complex manifold $X$.
\begin{Prop}\label{h-Adh}
For every $h\in R\setminus{\{0\}}$. Let $X$ be a compact complex $h$-$\partial\bar{\partial}$-manifold with $\dim_\C X=n$.
\begin{enumerate}
\item Every $d_hd_{-h^{-1}}$-cohomology class contains a $d_h$-closed representative.
\item The following map
$$
\begin{array}{ll}
G: &H^k_{h,A}(X,\C)\longrightarrow  H^k_{d_h}(X,\C)\\
&\hspace*{1cm} [\Omega]_{h,A}\longmapsto[\Omega]_{d_h}
\end{array}
$$
is well defined. Furthermore $G$ is an isomorphism.
\end{enumerate}
\end{Prop}

\begin{proof}
1. Let $\Omega$ be a $d_hd_{-h^{-1}}$-closed $k$-form. We want to show the existence of $(k-1)$-forms $u$ and $v$ such that 
$$d_h(\Omega+d_h u+d_{-h^{-1}} v)=0 \hspace*{2ex} \Leftrightarrow \hspace*{2ex}  d_h\Omega=-d_hd_{-h^{-1}} v. $$
It is clear that $d_h\Omega\in \mbox{ Im }d_h$ is $d_h$-closed and $d_{-h^{-1}}$-closed. Accordingly $d_h\Omega$ is $d_hd_{-h^{-1}}$-exact by the $h$-$\partial\bar{\partial}$-lemma that holds on $X$.\\

\noindent 2. Let $\Omega_1$ and $\Omega_2$ be  $d_h$-closed $k$-forms representing the same $h$-Aeppli cohomology class. So, there exist $(k-1)$-forms $\alpha$ and $\beta$ such that 
$$\Omega=\Omega_1-\Omega_2=d_h \alpha+d_{-h^{-1}} \beta .$$
Using the $h$-$\partial\bar{\partial}$ hypothesis on $X$, we obtain that $d_{-h^{-1}}\beta=\Omega-d_h\alpha\in \mbox{ Im } d_{-h^{-1}}$, which is $d_h$-closed and $d_{-h^{-1}}$-closed, is $d_h$-exact. This proves that the map $G$ is independent of choice of $d_h$-closed representative of the $h$-Aeppli class.\\
For any representative $u\in[\Omega]_{d_h}$, there exists a $(k-1)$-form $v$ satisfying $u=\Omega+d_h v$. Since $\Omega$ is $d_h$-closed, then $u$ is $d_hd_{-h^{-1}}$-closed. We obtain
$$[\Omega]_{d_h}=[\Omega+d_h v]_{d_h}=G([\Omega+d_h v]_{h,A}). $$
Thereupon $G$ is surjective. Since $\mbox{ Im }d_h\subset \mbox{ Im }d_h+\mbox{ Im }d_{-h^{-1}}$, it is straightforward to prove that $G$ is injective.
\end{proof}
An important tool to prove the main result of this section follows
\begin{The}\label{hp-HS p-SKT}
For every $h\in\R\setminus{\{0\}}$ an arbitrary constant. Let $X$ be a compact complex $h$-$\partial\bar{\partial}$-manifold with $ \dim_\C X=n$.
\begin{enumerate}
\item $X$ is an $hp$-HS manifold $\Longrightarrow$ $X$ is a $p$-HS manifold $\Longrightarrow $ $X$ is a $p$-SKT manifold.
\item  If $X$ is a $p$-SKT manifold, then $X$ is an $hp$-HS manifold.
\end{enumerate}
\end{The}
\begin{proof}
1. Suppose that there exist forms $\Omega^{i,2p-i}\in C^\infty_{i,2p-i}(X,\C)$ for $i=0,\cdots,p-1$ such that $d_h\tilde{\Omega}=d_h(\displaystyle{\sum_{i=0}^{p-1}\Omega^{i,2p-i}+\Omega+\sum_{i=0}^{p-1}\overline{\Omega^{i,2p-i}}})=0.$ Due to the proposition \ref{dhDR}, $\tilde{\Omega}$ is $d$-closed. Thus $\Omega$ is a $p$-HS form on $X$. Hence $\Omega$ is a $p$-SKT form on $X$ (cf. \cite{B19}).\\

\noindent 2. Consider that $X$ is a $p$-SKT manifold, i.e. there exists a weakly strictly positive $(p,p)$-form $\Omega$ satisfying the $\partial\bar{\partial}$-closedness. Using the decomposition (\ref{h-Aeppli decomp}), let $\tilde{\Omega}=\displaystyle{\sum_{i=0}^{p-1}\Omega^{i,2p-i}+\Omega+\sum_{i=0}^{p-1}\overline{\Omega^{i,2p-i}}}$ be the splitting of the $2p$-form $\tilde{\Omega}$. The Aeppli cohomology class $[\Omega]_A$ is the image of the $h$-Aeppli cohomology class $[\tilde{\Omega}]_{h,A}$ under the projection $H^{2p}_{h,A}(X,\C)\longrightarrow H^{p,p}_{A}(X,\C)$. The surjectivity of the projection implies that if $\Omega$ is a $\partial\bar{\partial}$-closed $(p,p)$-form, then there exists a $d_hd_{-h^{-1}}$-closed $2p$-form $\tilde{\Omega}$ such that $\Omega$ is its component of type $(p,p)$. By the proposition (\ref{h-Adh}), we can deduce that $\tilde{\Omega}$ is $d_h$-closed. Consequently, $\Omega$ is an $hp$-HS form on the $h$-$\partial\bar{\partial}$-manifold $X$.
\end{proof}
Another tool that we will need to prove theorem \ref{deform pSKT} is the following
\begin{The}\label{hpHS deform}
Let $(X_t)_{t\in\Delta}$ be a holomorphic family of compact complex manifolds. If $X_0$ is an $hp$-HS $h$-$\partial\bar{\partial}$-manifold for some $h\in\R\setminus\{0\}$, then $X_t$ is an $hp$-HS $h$-$\partial\bar{\partial}$-manifold for every $t\in\Delta$ close enough to $0$.
\end{The}
\begin{proof}
Suppose that $X_0$ is an $hp$-HS $h$-$\partial\bar{\partial}$-manifold. There exists an $hp$-HS form $\Omega$ on $X$, i.e. there exist $\Omega^{i,2p-i}\in C^\infty_{i,2p-i}(X,\C)$ for $i=0,\cdots,p-1$ such that 
$$d_h \tilde{\Omega}=d_h\left(\sum_{i=0}^{p-1}\Omega^{i,2p-i}+\Omega+\sum_{i=0}^{p-1}\overline{\Omega^{i,2p-i}}\right)=0. $$ 
$\tilde{\Omega}$ is a real $d_h$-closed $2p$-form on $X_0$. Moreover, since $X_0$ is an $h$-$\partial\bar{\partial}$-manifold and by the proposition \ref{dhDR}, we have $d\tilde{\Omega}=0$. Denoting by $(\Omega_t)_{t\in\Delta}$ the smooth family of component of $\tilde{\Omega}$ of type $(p,p)$ for the complex structure $J_t$ of $X_t$. The forms $\Omega_t$ vary in a $C^\infty$ way with $t\in\Delta$ for $t$ close to $0$. Additionally, the strict weak positivity of $\Omega_0$ implies the strict weak positivity of $\Omega_t$ for $t$ close to $0$ (cf. \cite{B19}, Lemma 4.2). Note by 
$$ \tilde{\Omega}_t=\sum_{i=0}^{p-1}\Omega^{i,2p-i}_t+\Omega_t+\sum_{i=0}^{p-1}\overline{\Omega^{i,2p-i}_t}$$
the $d$-closed $2p$-form on $X_t$ with $\Omega_t$ is its component of type $(p,p)$. Furthermore, the $h$-$\partial\bar{\partial}$-property being deformation-open (cf. \cite{BP18}) implies the existence of a $d_h$-closed representative for every $d$-cohomology class on $X_t$ (i.e. $d_h(\tilde{\Omega}_t+du)=0$ where $u$ is a $(2p-1)$-form and $d\tilde{\Omega}_t=0$). Hence, $\tilde{\Omega}_t$ is $d_h$-closed $2p$-form on $X_t$. As a result, $X_t$ is an $hp$-HS $h$-$\partial\bar{\partial}$-manifold for any $t$ close to 0.
\end{proof}
Now, one can show the following statement
\begin{The}\label{deform pSKT}
For every $h\in\R\setminus\{0\}$ an arbitrary constant. Let $\pi: \mathcal{X}\longmapsto \Delta$ be a holomorphic family of compact complex manifolds of dimension $n$ and $p\in\{0,\cdots,n\}$. If $X_0$ is a $p$-SKT $h$-$\partial\bar{\partial}$-manifold, then $X_t$ is a $p$-SKT $h$-$\partial\bar{\partial}$-manifold for every $t\in\Delta$, after possibly shrinking $\Delta$ about $0$.
\end{The}
\begin{proof}
If we suppose that $X_0$ is a $p$-SKT $h$-$\partial\bar{\partial}$-manifold, by theorem \ref{hp-HS p-SKT}, $X_0$ is an $hp$-HS $h$-$\partial\bar{\partial}$-manifold. While theorem \ref{hpHS deform} shows that $X_t$ is also an  $hp$-HS $h$-$\partial\bar{\partial}$-manifold for all $t\in \Delta$ close enough to 0. Using again theorem \ref{hp-HS p-SKT}, we conclude that $X_t$ is a $p$-SKT $h$-$\partial\bar{\partial}$-manifold for all such t's.
\end{proof}


\vspace{6ex}

\begin{Acknowledgements}
This work is supported by Tianyuan Mathematical Center in Southwest China and NSFC (Nos. 11890660 and 11890663). The author is grateful to this institution for providing her an excellent  working environment. She would like to thank in particular  Xiaojun Chen and Bohui Chen for assisting her and for useful discussions. She also thank Dan Popovici for suggesting her to study the local deformations of Calabi-Yau $\partial\bar{\partial}$-manifold that are co-polarised by the Gauduchon metric, for fruitful discussions and for reading this paper  and  Luis Ugarte for pointing out the investigation of the holomorphic deformations of the p-SKT property on a compact complex $\partial\bar{\partial}$-manifold.
\end{Acknowledgements}

\vspace{1ex}

\vspace{6ex}

\noindent Houda BELLITIR\\
Tianyuan Mathematical Center in Southwest China\\
Sichuan University\\
No. 24 South Section 1, First Loop Road,\\
 Chengdu, Sichuan Province P.R. China, 610065\\
houda.bellitir@yahoo.fr

\end{document}